\documentclass{amsart}
\usepackage{amsmath, amssymb,pdfsync}
 \usepackage{amscd}
\usepackage{mathrsfs}
\usepackage{tensor}
\usepackage{csquotes}

\theoremstyle{plain} 
\newtheorem{theorem}[subsection]{Theorem}
\newtheorem{proposition}[subsection]{Proposition}
\newtheorem{lemma}[subsection]{Lemma}
\newtheorem{corollary}[subsection]{Corollary}

\theoremstyle{definition}

\newtheorem{example}{Example}
\newtheorem{definition}[subsection]{Definition}

\theoremstyle{remark}
\newtheorem{remark}[subsection]{Remark}
\newtheorem*{remark*}{Remark}

\numberwithin{equation}{subsection}

\newcommand{\lsub}[2]{\tensor[_{#1}]{{#2}}{}}
\newcommand{\lrsub}[3]{\tensor[_{#1}]{{#2}}{_{#3}}}
\DeclareMathOperator{\Ob}{\mathrm{Ob}}
\DeclareMathOperator{\Mor}{\mathrm{Mor}}
\DeclareMathOperator{\cx}{\mathrm{cx}}
\DeclareMathOperator{\ex}{\mathrm{ex}}
\DeclareMathOperator{\cone}{\mathrm{cone}}
\newcommand{\dotcup}{{\,\dot\cup\,}}
\newcommand{\dslash}{/\negthinspace/}
\DeclareMathOperator{\PA}{\textup{\texttt{PA}}}
\DeclareMathOperator{\SGS}{\textup{\texttt{SGS}}}

\usepackage{color}
\definecolor{Mcolor}{rgb}{0,0,1}
\definecolor{Wcolor}{rgb}{1,0,0}
\definecolor{Ocolor}{rgb}{0,0.5,0}
\definecolor{lightgray}{rgb}{0.6,0.6,0.6}

\begin{document}

\title[groupoids with root systems]{A characterization of simplicial oriented geometries as groupoids with root systems}

\author{Matthew Dyer}
\address{Department of Mathematics\\ 255 Hurley Building \\ University of Notre Dame \\Notre Dame \\ Indiana 46556 \\U.S.A.}
\email{dyer@nd.edu}
\author{Weijia Wang}
\address{School of Mathematics (Zhuhai)
\\ Sun Yat-sen University \\
Zhuhai, Guangdong, 519082 \\ China}
\email{wangweij5@mail.sysu.edu.cn}

\begin{abstract}
This paper  shows that simplicial oriented geometries can be characterized as  groupoids with root systems having  certain favorable  properties, as   conjectured  by the first author. The proof  first translates  Handa's characterization of oriented matroids, as  acycloids which remain acycloids under iterated elementary contractions, into the language of groupoids with root systems, then  establishes favorable lattice theoretic properties of a generalization  of a construction which  Brink and Howlett used  in  their study of  normalizers of  parabolic subgroups of  Coxeter groups and uses Bj\"orner-Edelman-Ziegler's lattice theoretic characterization of simplicial oriented geometries amongst oriented geometries.
\end{abstract}

\maketitle

\section{introduction}

Groupoids with root systems have received  attention recently in   various  mathematical contexts. When studying the normalizer of a parabolic subgroup of a Coxeter group, Brink and Howlett \cite{bh}  considered    certain groupoids  which have presentations by generators and relations  resembling those of a Coxeter group
 (see \cite{Bour}, \cite{Hum} and \cite{bjornerbrenti} as standard references on Coxeter groups).  There are also  closely related  notions of Weyl groupoid and Coxeter groupoid developed by Cuntz, Heckenberger and Yamane (\cite{weylgroupoid}, \cite{weylgroupoid1}, \cite{weylgroupoid2}, \cite{weylgroupoid3}, \cite{weylgroupoid4}, \cite{weylgroupoid5}, \cite{weylgroupoid6}), which were initially   studied for their connections with certain Hopf algebras.

Motivated in part by longstanding conjectures (surveyed in \cite{DyW}) involving  a conjectural ortholattice completion of weak order of an (infinite) Coxeter group,  the first author defined in \cite{rootoid1} and \cite{rootoid2} the notions of protorootoid and rootoid by abstracting
lattice-theoretic  properties of weak order on Coxeter groups to a
  setting associated to  groupoids with root systems. This provides in
  particular a unified framework in which Weyl and Coxeter groupoids, and the groupoids of Brink and Howlett, can be studied, along with other mathematical structures including oriented matroids. One of the
   advantages of this framework is that various classes of rootoids are stable under natural categorical constructions, including formation of  categorical limits and functor rootoids. In particular,  the construction of
   Brink and Howlett  can be generalized and extended to the context of rootoids.

In this paper, we shall find it more convenient to work with the  concrete notion of signed groupoid set rather  than with  the  more abstract notion of  protorootoids (the relation between these two notions is closely analogous to that between Boolean algebras of sets and Boolean algebras). A signed groupoid set
 is a triple $(G,\Phi,\Phi^+)$ where $G$ is a groupoid and $\Phi=\{\lsub{a}{\Phi}\}$
 (called the root system) is a family of definitely involuted sets (i.e. sets with an involution map and a chosen ``positive'' part $\lsub{a}{\Phi}^+$) indexed by the objects of $G$ such that the groupoid acts on $\Phi$ with  action maps $\lrsub{a}{G}{b}\times \lsub{b}{\Phi}\rightarrow \lsub{a}{\Phi}$ where $\lrsub{a}{G}{b}:=\mathrm{Hom}_G
(b,a)$. Action by groupoid elements is required to commute with  involution maps but  is not required to preserve positive elements.  A Coxeter group, viewed as a  groupoid with one object,  and  with its standard root system is a typical example.

 Oriented matroids  are  combinatorial structures which abstract basic convexity  properties of sets of vectors in real vector spaces, or, from a dual point of view,  they abstract real, finite, essential hyperplane arrangements.  See \cite{OrMatBook}  as a general reference. By a simplicial oriented geometry, we mean a reduced simplicial oriented matroid without parallel elements; these abstract some  features of real,  finite, central,  simplicial hyperplane arrangements. The study of hyperplane arrangements and especially simplicial arrangements has been important  in  parts of algebra,  algebraic geometry, combinatorics, representation theory and topology over the last several decades; amongst fundamental  work in this area we mention \cite{del} and \cite{bes}.

 Any
 reduced oriented matroid can be naturally given the structure of a connected and simply connected signed groupoid set. The objects of  the  groupoid are in bijection with the hemispaces of the oriented matroid, which correspond to chambers of a hyperplane arrangement,  and the root system at a given object is precisely the original oriented matroid viewed as an involuted set  (this corresponds to the set of unit normal vectors to the hyperplanes, with involution given by multiplication by   $-1$). The positive roots at a given object are those elements in the hemispace represented by that object,
corresponding to roots with positive inner product with a vector in the corresponding open chamber. Each groupoid morphism acts trivially on the set of roots, but since the positive roots at its domain and codomain may differ, the morphism may change the signs of certain roots.

This work focuses on the problem of finding the properties that a  signed groupoid set $(G,\Phi,\Phi^+)$ needs to have so that it comes from a simplicial oriented geometry (i.e. a reduced simplicial oriented matroid
without parallel elements). We show that, as conjectured in \cite{rootoid2}, a simplicial oriented geometry can be characterized as a finite,
 connected, simply connected signed groupoid set which is    real, principal  and complete.  Here, ``finite'', ``connected'' and simply connected'' have standard meanings.  ``Real'' means every root has its sign changed by
some morphism. ``Principal'' means that the groupoid is generated by ``simple morphisms'' which each make a single positive root negative, and that the length of a groupoid morphism
with respect to the simple generators is equal to the cardinality of its inversion set (the set of positive roots made negative by the
element's inverse). Finally, ``complete'' means that the weak order at each object, given  by inclusion of inversion sets of morphisms with that object as codomain, is a complete lattice.  (These notions will be defined precisely in Section \ref{groupoidsetnotions}.)
  The fact that a simplicial oriented geometry gives rise to such a  signed groupoid set is essentially a reformulation of facts known from \cite{hyperplane}. The    key ingredients used to prove   the  other direction of the correspondence   are (1) Handa's characterization of oriented matroids  and  (2) a generalized Brink-Howlett construction described in this paper.  Handa's characterization of  oriented matroids  uses the notion of hemispaces and contraction.

Starting with a weaker notion called acycloid, one can perform contraction operations on it. Handa's theorem states that if after  all  sequences of elementary  contractions,  one still gets an acycloid, then the original acycloid is in fact an oriented matroid.
On the other hand when an acycloid is viewed as a signed groupoid set, the non-trivial  elementary contractions  (that is, those at non-loops)  closely  correspond to taking  a single, special  connected component of the signed groupoid set obtained by applying  the  generalized Brink-Howlett construction. Our main theorem in  Section   \ref{gbhc} asserts that if the original signed groupoid set is finite, connected, simply connected, preprincipal  and  complete, then   taking any component, not just a special one corresponding to an elementary contraction of the associated acycloid, produces another  signed groupoid set having exactly the same properties. Since both the original and constructed signed groupoid set  have the structure of an acycloid, Handa's theorem ensures that  the original signed groupoid set  comes from an oriented matroid, which can be seen to be simplicial on lattice-theoretic grounds (lattice theoretic arguments are also needed in establishing that Handa's result is applicable).

As a significant corollary, the simplicial  oriented  matroids (by which we mean those whose  associated simple oriented  matroid is a simplicial oriented geometry)   are preserved by a more general construction than contraction, which we call hypercontraction, corresponding to  taking an arbitrary component after applying the generalized Brink-Howlett construction. Hypercontraction is defined for arbitrary signed groupoid sets and is interesting for many other  classes of them  besides that of simplicial oriented matroids. We leave many natural questions about it open.

Several of the results we prove in Sections 2 and 3 of this paper are consequences of  more general ones in the theory of rootoids. For simplicity and brevity, we prove in this paper only a little more than needed to establish the connection  of these more general facts with oriented matroids, and   don't discuss in detail such  more general results in relation to the ones here.   We conclude the paper with some remarks   and open problems related directly to the content of this paper.
  More complete discussion of further developments   can be found in \cite{rootoid1}--\cite{rootoid2}  and their planned  sequels.

\section{Preliminaries}

\subsection{Oriented Matroids}

 There are many equivalent axioms for oriented matroids (see \cite{OrMatBook}).  We emphasize  their  characterization by closure operators  as given in  \cite{FL} and  \cite{largeconvex},
 using a formalization in terms of involuted sets instead of signed sets (see Remark  \ref{signedset}). Full details on oriented matroids can't be included here, and   the  reader unfamiliar with them may find it helpful  to look at Example \ref{coneexamp}  below for motivation while reading  our  discussion.

\subsection{} By an involuted set, we mean a pair $(E,*)$ where $E$ is a set and $*$ is an involution of $E$; that is, $*\colon E\to E$ is a function, denoted by  $x\mapsto x^*
$ for $x\in E$, satisfying $x^{**}=x$. We say that $E$ is \emph{strictly involuted} if the involution is fixed point free (that is,  $x^*\neq x$ for all $x\in E$).

Recall that  a \emph{closure operator} on a set $E$ is a function $c\colon \mathcal{P}(E)\to \mathcal{P}(E)$, (where $\mathcal{P}(E)$ is the power set of $E$) such that
(1) $ A\subseteq c(A)$ if $ A\subseteq E$, (2) $ c(A)\subseteq c(B)$ if $ A\subseteq B\subseteq E$ and  (3) $ c(c(A))=c(A)$ if $A\subseteq E$.
Subsets $F$ of $E$ which satisfy $C(F)=F$ are said to be \emph{$c$-closed} or  just closed.
 One easily checks that the intersection of a family of closed sets is closed.
We say $c$ is \emph{reduced} if $c(\emptyset)=\emptyset$.
Also,  $c$ is called  \emph{finitary} or \emph{of finite type} if
whenever $A\subseteq E$  and $x\in c(A)$ there exists a finite set $B\subseteq A$ such that $x\in c(B)$.  Given a (finitary) closure operator $c$ on $E$, for any disjoint  $A,B\subseteq E$, one has a
(finitary) closure operator $c_{A,B}$ on $B$ given by $c_{A,B}(X)=c(A\cup X)\cap B$, for $X\subseteq B$. We call  $c_{A,B}$ a \emph{contraction} of $c$ (by $A$) if $B=E\setminus A$, a \emph{restriction} of $c$ (to $B$) if $A=\emptyset$ and a \emph{minor} of $c$ in general.

\subsection{} An \emph{oriented matroid} is a system $(E,*,\cx)$  where $E$ is a set with a map $*: E\rightarrow E$ and a map $\cx: \mathcal{P}(E)\rightarrow \mathcal{P}(E)$  such that

(M1) $(E,*)$ is a strictly\footnote{   \label{f1}  A largely equivalent theory may be developed in which ``strictly'' is dropped from (M1) (see \cite{largeconvex}), but many   statements and definitions   then become more cumbersome (for example, in the definition of proper circuit, one then has to add a condition that $e\neq e^{*}$), while only a few become more natural (see footnote \ref{f3} to Example \ref{coneexamp}). Similarly, the definitions of (pre)acycloids and signed groupoid sets given later may  be modified to allow non-strictly involuted ground sets.}   involuted set,

(M2) $\cx$ is a closure operator  on $E$,

(M3)  $\cx$ is finitary,

(M4) $\cx(X)^*=\cx(X^*)$  for all $X\subseteq E$,

(M5) if $X\subseteq E$ and  $x\in \cx(X\cup \{x^*\})$, then $x\in \cx(X)$,

(M6)
if $X\subseteq E$ and  $x,y\in E$ with  $x\in \cx(X\cup \{y^*\})$ but $x\not\in \cx(X)$, then\footnote{The conclusion $y\in \cx((X\setminus\{y\})\cup\{x^*\})$ in (M6) is routinely  misstated in the  literature as $y\in \cx((X\cup\{x^*\})\setminus\{y\})$;  that version would imply  that $\cx(A)=E$ for all $A\subseteq E$.}  $y\in \cx((X\setminus\{y\})\cup\{x^*\})$.

 We remark at once that if $A$ and $B$ are disjoint subsets of $E$ satisfying $A=A^{*}$ and $B=B^{*}$, then
$(B,\dag,\cx_{A,B})$ is an oriented matroid, where $\dag$ is the restriction of $*$ to an involution on $B$ and $\cx_{A,B}$ is the minor of $\cx$ from $A$ and $B$.

When the closure operator and the involution map are understood, we usually denote an oriented matroid by $E$ for simplicity. We recall some facts   about an alternative description of oriented matroids using the concept of their  circuits (see for instance
\cite{FL} and  \cite{largeconvex}).  A \emph{circuit} of $(E,*,\cx)$ is a minimal nonempty   subset $X$ of $E$  such that $X^*\subseteq \cx(X)$. A circuit is called \emph{improper} if it is of the form $\{e,e^*\}$ for some $e\in E$. A circuit $C$ which is not improper is said to be \emph{proper}, and satisfies  $C\cap C^*=\emptyset$.  The set of circuits (together with $*$ map) determines the oriented matroid by requiring
   \begin{equation*}
\cx(X)=\{e\in E\mid \text{\rm there exists $U\subseteq X$  such that $ U\cup \{e^*\}$    is a circuit}\}.
\end{equation*} Oriented matroids may be axiomatized in terms of their circuits and involution $*$.

\subsection{}  We say a set $F\subseteq E$ is \emph{closed} if  $F$ is $\cx$-closed.   The elements of  the closed set $L:=\cx(\emptyset)=L^*$ are called \emph{loops} of $E$. We say $E$ is \emph{reduced} if  $\cx$ is reduced (that is, $E$ has no loops, or $L=\emptyset$). We
say $E$ is \emph{simple} if it is reduced and all singleton subsets of $E$ are closed.   An oriented matroid $(E,*,\cx)$ has an \emph{associated reduced oriented matroid} $(E_0 ,\dag,\cx_0)$ where $E_0:=E\setminus L\subseteq E$, $e^\dag:=e^*$ for $e\in E_0
$  and $\cx_0(X):=\cx(X)\setminus L$ for $X\subseteq E_0$. There is also a  simple oriented matroid $(E_1,\ddag,\cx_1)$ associated to $(E,*,\cx)$,  defined as follows.
There is an equivalence relation $\sim$ on $E_0$ such that $e\sim f\iff \cx_0(e)=\cx_{0}(f)$, for $e,f\in E_0$. Let $E_1:=E_0/\sim$, $\ddag([e]):=[e^\dag]$ for $e\in E_0 $, and $\cx_1(X)=\cx(\bigcup_{x\in X} x)/\sim$ for $X\subseteq E_1$, where $[e]$ denotes the $\sim$-equivalence class of $e\in E_0$.
(In general, for any equivalence relation $\equiv$ on a set $X$, we write
 $Y/\equiv$ for the set of $\equiv$-equivalence classes which are contained in $Y$, for any subset $Y$ of $X$ which is a union of $\equiv$-equivalence classes.)

Using the  relations between $E$ and $E_0$,  we extend some terminology used in \cite{largeconvex} from reduced oriented matroids   to general oriented matroids.
A closed set $F$ is called a \emph{sharp} if $F\cap F^*=L$.
A \emph{hemispace} of   $E$ is a closed set $H$ such that $E=H^*\cup H$ and $H\cap H^*=L$.    Hemispaces are the same as \emph{maximal sharps} (that is, inclusion maximal elements of the set of sharps) and  any sharp, such as $L$, is contained in a hemispace; see \cite[Theorem 7, Proposition 8]{largeconvex}.
We also define a \emph{tope} of $E$ to be a set of the form
$H\setminus L$ for some  hemispace $H$ of $E$.  (Therefore if $L=E$ there exists a unique tope, i.e. $\emptyset$.)
   Using the circuit axioms for  oriented matroids in \cite{largeconvex} (or in \cite{hyperplane}), it can be shown that  the circuits are the  inclusion-wise minimal elements  in $\mathcal{P}(E)$ which are contained in no tope (\cite[Theorem 1.1]{radon}), and the topes are  the inclusion-wise  maximal  elements in $\mathcal{P}(E)$   which contain no circuit.

\begin{lemma}\label{nosubcircuit}
Let $M=(E,*,\cx)$ be a reduced oriented matroid. Suppose that $C$ is a  circuit of $M.$ Let $x\in C$. Then $(C\backslash \{x\})\cup \{x^*\}$ contains no circuit.
\end{lemma}

\begin{proof}
Note that since $M$ is reduced, $C\backslash \{x\}$ is not empty.
Suppose to the contrary, $C'$ is a circuit contained in $(C\backslash \{x\})\cup \{x^*\}$. By the circuit axioms of an oriented matroid (see \cite[Section 6 axiom (C1)]{hyperplane}), a circuit cannot be properly contained in another circuit. So $C'$ cannot be contained in $C\backslash \{x\}$. Therefore $C'=D\dotcup \{x^*\}$ where $D\subseteq C\backslash \{x\}$.  But the circuit axioms of an oriented matroid ensures that there exists a circuit $Z$ contained in $(C\backslash \{x\})\cup D=C\backslash \{x\}$  (see \cite[Section 6 axiom (C3)]{hyperplane}). This is a contradiction.
\end{proof}

We record the following explicit  description of  oriented matroid  closure operators  in terms of  hemispaces.
\begin{theorem}\label{closure}
Let $M=(E,*,\cx)$ be an oriented matroid with the set $\mathfrak{H}$ of hemispaces. (Note $\mathfrak{H}=\{T\cup L\mid T\in \mathfrak{T}\}$ where $\mathfrak{T}$ is the set of topes of $M$ and $L=\cx(\emptyset)$ is the set of loops.)  For $H\in \mathfrak{H}$,  let $\cx_{H}$ denote the restriction of $\cx$ to a closure operator on $H$. That is, for $X\subseteq H$, $\cx_{H}(X)= \cx(X)\subseteq H$.

(a) For $H\in \mathfrak{H}$ and $X\subseteq H$, we have
$\cx_{H}(X)=\bigcap_{\substack{K\in \mathfrak{H}\\K\supseteq  X}}K$.

(b) The closure operator $\cx$  is given by
$\cx(X)=\bigcup_{H\in \mathfrak{H}}\cx_{H}(X\cap H)$.
\end{theorem}

\begin{proof}  Let $L:=\cx(\emptyset)$ be the set of loops.
Note  $L$ is contained in each closed set, including each hemispace. By considering the associated reduced oriented matroid,  one easily reduces to the case $L=\emptyset$.  So henceforth we assume $M$ is reduced.

(a)
Let $X\subseteq H\in \mathfrak{H}$. Then
$\cx_{H}(X)=cx(X)\subseteq \bigcap_{K\in \mathfrak{T},K\supseteq  X}K$ since the right hand side is closed  in $E$  and contains $X$.
 For the reverse inclusion, let $X\subseteq H$ and $ x\in E$. Assume that $X$ and $x$ satisfy the property that  for any hemispace $K$, if  $K\supseteq  X$ then $x$ must also be in $K$. Then we claim that there exists some $V\subseteq X$ such that  $V\cup \{x^*\}$ is a circuit. Otherwise no subset of $X\cup \{x^*\}$ is a circuit. So $X\cup \{x^*\}$ must be contained in some hemispace $K$   by the discussion of circuits and hemispaces above. But $x\in K$ also, which contradicts that $K\cap K^*=\emptyset$. Hence $x\in \cx(X)=\cx_{H}(X)$, and (a) is proved.

(b)  Let $X\subseteq E$ be arbitrary. Then clearly
$\bigcup_{H\in \mathfrak{H}}\cx_{H}(X\cap H)\subseteq \cx(X)$.
  Now take $x\in \cx(X)$. Then there exists  $X'\subseteq X$ such that $X'\cup \{x^*\}$ is a circuit. This implies that $X'\cup \{x\}$ is contained in a hemispace $H$ by Lemma \ref{nosubcircuit} and the discussion of circuits and hemispaces above.  So $x\in \cx(X')\subseteq \cx(X\cap H)=\cx_{H}(X\cap H)$,  so we are done.
\end{proof}

\subsection{}\label{convgeom}  An \emph{anti-exchange} (or \emph{convex}) closure operator on a set $U$ is a closure operator $c$ on $U$ such that   for $p,q\in U$ and  $X\subseteq U$, if $q\in c(X\cup\{p\})$ but $q\not\in  c(X)$, then $p\not\in c (X\cup\{q\})$.  For any disjoint $A,B\subseteq U$, the minor $c_{A,B}$ is then  an anti-exchange closure on $B$.
Let us say that a subset $P$ of a poset $Q$ is \emph{saturated} in $Q$ if every maximal chain of $P$ (regarded as subposet of $Q$) is a maximal chain of $Q$.
If $Q$ is the Boolean poset of all subsets of a set $U$, then
 the maximal chains of $Q$ are in natural bijection with total orders of $U$, by a map sending each total order to its set of downsets.    In that case, for any saturated subposet $P$ of $Q$, the map $X\mapsto \bigcap_{Y\in P,Y\supseteq  X}Y$ is an anti-exchange closure operator on $U$.
   On the other hand,  it can  be shown that if $c$ is a reduced,  finitary anti-exchange closure on $U$, then the set $P$ of $c$-closed subsets of $U$ is saturated in the power set $Q$ of  $U$ and the anti-exchange closure on $U$ from $P$ is just $c$.

      Suppose   that   $M$ is a simple  oriented matroid and that $H\in \mathfrak{H}$ is a hemispace.  It can be shown that  the poset $P=\{H\cap K\mid K\in \mathfrak{H}\}$ is saturated in the power set $Q$ of $H$. By Theorem \ref{closure},   the associated
  anti-exchange     closure operator on $H$ is
           $\cx_{H}$ (which is also  finitary).    All these facts are well known for finite $E$ (see \cite{EdJam},  \cite{Ed} and \cite{hyperplane}).

\subsection{}  Recall (see for example \cite{MatBook}) that a (finitary, possibly infinite, unoriented) \emph{matroid}  is a pair $(F,c)$ where $F$ is a set and $c$ is a finitary closure operator on $F$ such that  the following \emph{exchange condition} holds: if $X\subseteq F$ and  $x,y\in F$ satisfy
$x\in c(X\cup\{y\})\setminus c(X)$, then $y\in c(X\cup\{x\})$.
For any disjoint sets $A,B\subseteq F$, the minor $(B,c_{A,B})$ is also a matroid .The \emph{rank}
of $(F,c)$ is  the (well-defined) cardinality of any  subset $B$ of $F$ which is inclusion minimal subject to $c(B)=F$ (such a set is called a \emph{basis} of $F$).

 For each oriented matroid $(E,*,\cx)$ one can associate a  matroid $(E,\overline{\cx})$  to it where $\overline{\cx}(X)=\cx(X\cup X^{*})$ for $X\subseteq E$.

\subsection{}\label{extreme} In this subsection,  we  assume the oriented matroid is finite, i.e. $\vert E\vert <\infty$.    An \emph{oriented geometry} is a finite, simple oriented matroid.
  Let $(E,*,\cx)$ be an oriented geometry and $H$ be a hemispace.  There exists a unique minimal subset of $H$, called the set of \emph{extreme elements} of $H$, and denoted  by  $\ex (H)$, such that $\cx(\ex (H))=H$ (see \cite[Section 6]{hyperplane}; it follows from the fact $\cx_{H}$ is anti-exchange). If $\vert \ex (H)\vert $ is equal to the rank of  the  unoriented matroid  $(E,\overline{\cx})$, then $H$ is called \emph{simplicial}. If all hemispaces are simplicial then we call the oriented geometry a \emph{simplicial oriented geometry}.
 By a \emph{simplicial oriented matroid}, we mean an oriented matroid for which the associated  simple oriented matroid  is  a simplicial oriented geometry.

\begin{example}\label{coneexamp} Let $V$ be a real vector space. For any $A\subseteq V$, define \begin{equation*}
\cone (A):=\{\sum_{i\in I}k_iv_i\mid v_i\in A, k_i\in \mathbb{R}_{\geq 0}, \vert I\vert <\infty\}
\end{equation*} where by convention the empty sum has value $0\in V$. Thus, $\cone(A)$ is the pointed convex cone in $V$ spanned by $A$. Say that a subset $A$ of $V$ is positively independent if
$\sum_{i\in I}k_iv_i=0$ with $v_i\in A$, all $k_i\in \mathbb{R}_{\geq 0}$ and  $\vert I\vert <\infty$ implies $k_{i}=0$ for all $i\in I$.

For any subset $E=-E$ of $V\setminus \{0\}$, there is
a reduced oriented matroid\footnote{\label{f3} If we had not required $E$ to be strictly involuted, we could have allowed more generally $E=-E\subseteq V$ and all such oriented matroids would be minors of the one with $E=V$} $M_{E}:=(E,*,\cx_{E})$ with $\cx_{E}(A):=\cone(A)\cap E$ and $x^{*}=-x$ for $x\in E$.  (Examples of non-reduced oriented matroids arise as minors of these.)
The  oriented  matroid $M_{E}$ is simple if and only if for  $\alpha\in E$ and $0\neq c\in \mathbb{R}$, one has $c\alpha\in E$ if and only if $c=\pm 1$.

 For $X\subseteq E$, we have $\overline{cx}(X)=\mathrm{span}(X)\cap E$, where  $\mathrm{span}$ denotes linear span. The hemispaces  of $M_{V\setminus\{0\}} $ are the sets of  positive elements of vector space total orderings of $V$, and the hemispaces of $M_{E}$ are the intersections of such hemispaces with $E$. The circuits of $M_{E}$ are the minimal non-empty subsets of $E$ which are not  positively independent. Suppose for example that
the hemispace $H$ of $M_{E}$ is contained in some affine subspace $U$ (that is, a translate of a codimension one linear  subspace of $V$) such that $0\not \in U$.  Then the convex closure operator $\cx_{H}$ from $M_{E}$ is given by
$\cx_{H}(X)=\mathrm{conv}(X)\cap H$ for $X\subseteq H$, where $\mathrm{conv}(X)$ denotes the convex hull of a subset of  $U$.

 Now suppose that $M_{E}$ is an oriented geometry. Let $V_{0}=\mathrm{span}(E)$. This is a finite-dimensional real vector
 space. Choose a positive definite inner product $(-\mid -)\colon V_{0}\times V_{0}\to \mathbb{R}$ on $V_{0}$. Consider  the set $A$ of linear hyperplanes of $V_{0}$ orthogonal to the
  elements of $E$. Then $A$ is a  real,  finite, essential, linear hyperplane arrangement in $V_{0}$ (essential means that $\bigcap_{H\in A}H=\{0\}$).  The connected components
  $V\setminus \bigcap_{H\in A}H$ are
  called (open) chambers. Every chamber $C$ determines a
  hemispace $H$ of $E$, consisting of all $\alpha\in E$ which have  positive inner product with some point (or equivalently,  all points) of $C$, and every hemispace so arises.   In fact, consider a hemispace $H$. It is positively independent and spans $V_{0}$, and $\ex(H)$ is a set of representatives of the extreme rays  of  the polyhedral cone $K=\cone(H)$ spanned by $H$.   Then the dual cone $K^{\vee}:=\{\alpha\in V_{0}\mid (\alpha\mid K)\subseteq \mathbb{R}_{\geq 0}\}$ also spans $V_{0}$ and its interior is a chamber yielding the hemispace $H$. Moreover,   the facets
  (codimension one faces) of  $K^{\vee}$ are in bijection with the elements of $\ex(H)$, so $K^{\vee}$ is simplicial if and only if
  $\vert\ex(H)\vert =\dim(V_{0})$, which is the rank of $E$.
  Thus, the hyperplane arrangement is simplicial (that is, all its chambers are open simplicial cones) if and only if $E$ is a simplicial oriented geometry.

  Oriented matroids and  geometries arise in many other, quite different ways, and those  arising as above will be said in this paper  to be realizable.  Not every oriented geometry is realizable; see \cite{OrMatBook}
  for more precise discussion of realizability and  similar models (for instance, pseudosphere arrangements) which can be used to describe finite oriented matroids in  general.
  \end{example}

\begin{remark}\label{signedset}  By definition, a signed subset (or signed  vector)  of a set $E'$ is a function  $f\colon E'\to \{+,0,-\}$. Associated to $E'$, one has a set $E:=E'\times \{\pm\}$
with a fixed point free involution $*$ defined by $(e,\pm)^*=(e,\mp)$ for $e\in E'$. Then the signed subsets of $E'$ correspond bijectively to the subsets $X$ of $E$ satisfying $X\cap X^*=\emptyset$, by a standard  correspondence which attaches to a signed subset $f$ of $E'$ the subset $\{(e,f(e))\in E\mid  e\in E', f(e)\neq 0\}$.  In much of the  literature,  oriented matroids, acycloids etc  are  defined by specifying  certain collections of  signed subsets of a set $E'$ (circuits, vectors, covectors, topes etc).  In this paper, we always  regard the sets of  circuits, vectors, covectors, topes etc  as  subsets of  $\{X\subseteq E \mid X\cap X^*=\emptyset\}$ for a strictly involuted set $E$,  using the above standard correspondence.  In particular, the proper circuits and topes of  an oriented matroid $E$ in the framework of strictly involuted sets as considered above correspond to the circuits and  topes, respectively, for oriented matroids in the framework of signed sets. \end{remark}

  We state  the following result more generally than needed in this paper, for  application elsewhere.

\begin{proposition}\label{hemiequiv} Let $M=(E,*,\cx)$ be an oriented matroid,  and $F$ be a subset of $E$ with $F=F^{*}$.   Let $M_{F}$ denote the restriction of $M$ to $F$.  Let  $L:=\cx(\emptyset)$ be the set of loops of $M$, so $L\cap F$ is  the set of loops of $M_{F}$. Then the hemispaces of  $M_{F}$ are  the intersections of $F$ with the hemispaces of $M$.  More precisely,  the following conditions are equivalent for $A\subseteq F$.

(i) $A$ is a hemispace of $F$ (that is, $F=A\cup A^{*}$, $ A\cap A^{*}= L\cap F$ and $\cx(A)\cap F=A$).

(ii) $F=A\cup A^{*}$,
$ L\cap F\subseteq A$ and there is no circuit of  $E$ contained in $A\setminus L$.

(iii) $A=H\cap  F$ for some hemispace  $H$ of $E$.

(iv) $A\cup A^{*}=F$, $\cx(A)$ is a sharp of $E$ and  $L\cap F\subseteq A$.
\end{proposition}
\begin{proof} By definition, $M_{F}:=(F,\dag, d)$  where $d$ is the restriction $d=c_{\emptyset, F}$  of $c$ (so $d(X)=c(X)\cap F$ for $X\subseteq F$),  and $x^{\dag }=x^{*}$ for $x\in F$).
Thus, hemispaces of $M_{F}$ are as described in (i). We note for the proof that $(L\cap F)^{*}=L^{*}\cap F^{*}=L\cap F$.

 We show that (i) implies (ii).  Assume that (i)  holds.
Suppose, contrary to (ii), that $\{a_{1},\ldots, a_{n}\}$ is a circuit  of $E$ which is contained in $A\setminus L$. Then
$a_{n}^{*}\in \cx(\{a_{1},\ldots, a_{n-1}\})\cap A^{*}\subseteq \cx(A)\cap F=A$. Hence $a_{n}\in A\cap A^{*}=L\cap F$, a contradiction.

Now we show (ii) implies (iii).  Assume (ii) holds. Since $A\setminus L$
contains no circuit  of $E$, it is contained in some tope $T$, and hence $A$  is contained in some hemispace $H=T\cup L$ of $E$.
Then $L\cap F\subseteq A\subseteq H\cap F$. We have
\begin{equation*}
H\cap A^{*}=(H^{*}\cap A)^{*}=(H^{*}\cap H\cap A)^{*}\subseteq(L\cap F)^{*}=L\cap F\subseteq A
\end{equation*}
 and hence
\begin{equation*}A\subseteq  H\cap F=H\cap (A\cup A^{*})=
(H\cap A)\cup (H\cap A^{*})\subseteq A,\end{equation*}
so $A=H\cap F$ as required.

To show (iii) implies (iv), assume that $A=H\cap F$ where
$H$ is a hemispace of $E$. Then $A^{*}=H^{*}\cap F^{*}=H^{*}\cap F$ so
\begin{equation*}A\cup A^{*}=(H\cap F)\cup (H^{*}\cap F)=(H\cup H^{*})\cap F=E\cap F=F
\end{equation*} and $A\supseteq A\cap A^{*}=(H\cap F)\cap (H^{*}\cap F)=L\cap F$.
 Suppose that $x\in (\cx(A)\cap \cx(A)^*)\setminus L$. Since $A^*=H^*\cap F$ and $\cx(A)^*=\cx(A^*)$, one has $x\in (H\cap H^*)\setminus L.$ This contradicts the fact that $H$ is a sharp. Hence $\cx(A)$ is a sharp of $E$.

Finally,  assume that (iv) holds. We show (i) follows. Since $L\cap F=(L\cap F)^{*}\subseteq A$, we have  \begin{equation*}
L\cap F\subseteq A\cap A^{*}\subseteq \cx(A)\cap A^{*}\subseteq \cx(A)\cap \cx(A)^{*}\cap F=L\cap F.
\end{equation*} Hence   $A\cap A^{*}=cx(A)\cap A^{*}=L\cap F$ and so
 \begin{equation*}
\cx(A)\cap F=\cx(A)\cap (A\cup A^{*})=(\cx(A)\cap A)\cup (\cx(A)\cap A^{*})=A\cup (A\cap A^*)= A,
\end{equation*}
completing the proof.
\end{proof}

We conclude our discussion of oriented matroids with the following technical fact, to  be used  in discussing realizations of signed groupoid sets in Section 4.
\begin{lemma}\label{firstcondition}
Let $(E,*,\cx)$ and $(F,-,c)$ be oriented matroids and $f\colon E\to F$ be an injective map  such that  $f(\alpha^{*})=-f(\alpha)$ for any  $\alpha\in E$  and $f(\cx(\emptyset))=f(E)\cap c(\emptyset)$ (note this latter condition holds if $E$ and $F$ are both reduced). Suppose that for all $A\subseteq E$ such that $A\cup A^{*}=E$ and $A\cap A^{*}=\cx(\emptyset)$, one has $c(f(A))\cap f(E)=f(A)$  if and only if $\cx(A)=A$ (that is, if and only if $A$ is a hemispace of $E$). Then for any $X\subseteq E$, we have
$f(\cx(X))=c(f(X))\cap f(E)$. In other words, $f$ induces an isomorphism of oriented matroids from $E$  to the restriction of $F$ to $f(E)$.\end{lemma}
\begin{proof} Note that for $A\subseteq E$, one has
$A\cup A^{*}=E$ and $A\cap A^{*}=\cx(\emptyset)$  if and only if
$f(A)\cup -f(A)=f(E)$ and $f(A)\cap -f(A)=f(\cx(\emptyset))$.
Thus, the assumptions together with Proposition  \ref{hemiequiv}
imply that $f$ induces a bijection $A\mapsto f(A)$ between the hemispaces of $E$ and those of the restriction of $F$ to $f(E)$. Since the hemispaces of an oriented matroid determine its closure operator
(by our discussion of the relation between hemispaces, topes, circuits and the closure operator,  or more concretely from Theorem \ref{closure}), the result follows. \end{proof}

\subsection{Acycloids}
In this paper we also use the notion of an  acycloid (see \cite{FH},
\cite{topechar}) which is weaker than that of   an  oriented matroid.

 We define a notion of a preacycloid.
A \emph{preacycloid} is a system $A=(E,*,\mathfrak{T})$ where $E$ is  a finite set with a map $*: E\rightarrow E$ and $\mathfrak{T}$ is a   collection of subsets of $E$ such that

(A1) $(E,*)$ is a strictly involuted set,

(A2)  if $H\in \mathfrak{T}$,   then  $H\cap H^*=\emptyset$ and  $H\cup H^*=E\setminus L$, where $L:=E\setminus \bigcup_{H'\in \mathfrak{T}}H'$,

(A3) if $H\in \mathfrak{T}$, then $H^*\in \mathfrak{T}$.

We call  the set $L$ in (A2)  the set of \emph{loops} of $A$. We say $(E,*,\mathfrak{T})$ is \emph{loopless} if
$L=\emptyset$.  By an element $e$ of $A$, we mean an element $e$ of the underlying set $E$ of $A$. We shall  call an element of $\mathfrak{T}$ of a preacycloid a  tope.

 We say that two elements $e$, $f$  of $E$ are \emph{parallel}  if for all $H\in \mathfrak{T}$, one has $e\in H\iff f\in H$.  Parallelism is an equivalence relation, which we denote as $\sim$ or $\sim_{A}$, on $E$. We denote the parallelism (that is, $\sim$-equivalence) class of $e\in E$ by $[e]$ or $[e]_{A}$. If $e$ is a loop, then $[e]=L\supseteq \{e,e^*\}$ so $\vert [e]\vert >1$.

An \emph{acycloid} is defined to be  a preacycloid satisfying the following  condition:

(A4)  $\mathfrak{T}\neq \emptyset$ and   if $H_1,H_2\in \mathfrak{T}$ with $H_1\neq H_2$, then there exists $e\in H_1\setminus H_2$  such that $(H_1\setminus [e])\cup [e]^*\in \mathfrak{T}$.

 We say that a   preacycloid (or acycloid)  as above  is  \emph{simple}   if it satisfies the  equivalent  conditions (A5)--(A5)' below:

(A5) every parallelism class in $E$ is a singleton set.

(A5)'  if $e,f\in E$ with $e\neq f$,  then there   exists $H\in \mathfrak{T}$ such that $e\in H$ and $f\not\in H$.

 Note that any simple preacycloid is loopless, by the above remarks on the parallelism class of a loop.

Any preacycloid $A=(E,*,\mathfrak{T})$ determines a simple preacycloid
$A^{\circ }:=(E',\dag,\mathfrak{T}')$, which we call the \emph{simplification}  of  $(E,*,\mathfrak{T})$, where $E':=(E\setminus L)/\negthinspace\sim$,
$[e]^\dag=[e^*]$ for  $e\in E\setminus L$ and $\mathfrak{T}'=\{H/\negthinspace\sim\,\mid H\in \mathfrak{T}\}$. It is easily seen that a preacycloid  $(E,*,\mathfrak{T})$ is an acycloid if and only if its simplification   is a simple acycloid.

\begin{lemma}\label{acycpath} Let $A$ be an acycloid and $H,H'$ be topes of $A$. Let $n$ be the number of parallelism classes contained in $H\setminus H'$. Then there is a sequence $H=H_{0},\ldots, H_{n}=H'$ of topes of $A$ and elements $e_{i}\in H_{i-1}\setminus H_{i}$  such that
$H_{i}=(H_{i-1}\setminus [e_{i}])\cup [e_{i}]^{*}$ for  $i=1,\ldots, n$.
\end{lemma}
\begin{proof} Note that $H\setminus H'$ is in fact a finite union of parallelism classes (of non-loops) and that $e_{i}\in H\setminus H'$ implies  $[e_i]\subseteq H\setminus H'$. The result  follows easily by induction on $n$ using the acycloid axioms. Details are omitted.\end{proof}

\begin{lemma}\label{matroidisacycloid}
 Let $M=(E,*,\cx)$ be a finite  oriented matroid and $\mathfrak{T}$ denote  the set of topes of $E$. Then

(a)  $A:=(E,*,\mathfrak{T})$ is an  acycloid,  called  the \emph{tope acycloid}  (or \emph{preacycloid}) of $M$.

(b)  $M$ is an oriented geometry if and only if  its tope acycloid  is  simple.
\end{lemma}

\begin{proof}
This is well known; for instance, it underlies the definition of (simple) acycloids (see \cite{topechar}, \cite{OrMatBook}). We sketch a proof for completeness.

  We first show that $A$ is a preacycloid; this is essentially trivial.  Condition (A1) follows from (M1), and (A3) is easily checked using (M4).  If $H\in \mathcal{T}$, then $H\cap H^*
=\emptyset$ and $H\cup H^*=E\setminus \cx(\emptyset)=E$.  Hence  (A2) holds, with $L=\cx(\emptyset)$.

We next show that if $M$ is an oriented geometry, then
$A$ is a simple acycloid.

To prove (A4), notice that $(H_1\setminus H_2)\cap \ex (H_1)\neq \emptyset$. To see this note that otherwise $\ex (H_1)\subseteq H_1\cap H_2\subsetneqq H_1$ and $H_1\cap H_2$ is a closed set, contradicting $\cx(\ex (H_1))=H_1$. Then take $e\in (H_1\setminus H_2)\cap \ex (H_1)$.  We claim that $\cx(H_1\setminus \{e\})=H_1\setminus \{e\}$ because otherwise we will have another minimal subset whose closure is $H_1$, contradicting the uniqueness of $\ex (H)$. Therefore  $H_1\setminus \{e\}$ is closed. So by
 \cite[Theorem 7]{largeconvex}
, there exists a hemispace $H_3$ such that
 $H_3\supseteq  H_1\setminus\{e\}$ and $e\not\in H_3$. Then necessarily, $H_3=(H_1\setminus \{e\})\cup\{e^*\}$ as required.

 We prove (A5).  Take
 $e,f\in E$ with $e\neq f$ and  $e\neq f^*$.  We claim  that
$\cx(\{e,f^*\})$ is a sharp.  Note that $e^*\not\in \cx(\{e,f^*\})$. To see this, note that $e^*\in \cx(\{e,f^*\})$ will imply $e^*\in \cx(\{f^*\})$ by (M5). But  singleton subsets are  closed in a simplicial oriented geometry, so this is a contradiction.
 Similarly, $f\not\in \cx(\{e,f^*\})$.
Take $x\not\in \{e,f,e^*,f^*\}$. We prove that it cannot happen that $\{x,x^*\}\subseteq \cx(\{e,f^*\})$.  For  if so, by (M6), we will have $f\in \cx(\{e,x^*\})$ and $f\in \cx(\{e,x\})$. That will force $f\in \cx(\{e\})$ by  \cite[Proposition 2]{largeconvex},  contrary to $\cx(\{e\})=\{e\}$. Therefore  $\cx(\{e,f^*\})$ is a sharp. By  \cite[Theorem 7]{largeconvex},  this sharp  is contained in a hemispace. Hence (A5) follows.
Now $A$ has no loops since $M$ has no loops. Lemma \ref{acycpath} then easily implies that all parallelism classes in $A$ are singletons, and $A$ is simple as required.

Now  suppose more generally that $M$ is just an oriented matroid, and let $M'$ be the corresponding oriented geometry.
It is easy to see that the  tope acycloid $A'$ of $M'$ is the simplification of the tope preacycloid $A$ of $M$, and $A$ is therefore an acycloid. Finally, suppose that $A$ is simple.
Since $A$ has no loops, $M$ has no loops.  Since $A'$ has singleton parallelism classes, the     parallelism classes  in $A$ must be  precisely the sets $\cx(\{e\})$ for $e\in E$ which are identified to singletons in forming $M'$ from $M$. It follows that if  $A$ is simple, these sets are singletons and so $M$ is an oriented geometry.
 \end{proof}

 Note that there exist    acycloids which are not tope acycloids   of oriented matroids;   see \cite{topechar}.

 \subsection{Contraction of acycloids}\label{cont} We consider some  constructions which preserve preacycloids. Let $A=(E,*,\mathfrak{T})$ be a preacycloid.
  If $\Gamma$ is any subset of $E$, define \begin{equation*}
\mathfrak{T}_{\Gamma}:=\{H\setminus \Gamma\mid \text{ \rm $ H\in \mathfrak{T}$, $\Gamma\subseteq H$ and $(H\setminus \Gamma)\cup \Gamma^{*}\in \mathfrak{T}$}\}.
\end{equation*}  Trivially,   if
$\Gamma$ is not a union  of parallelism classes of non-loops  or
if  $\Gamma\cap \Gamma^{*}\neq \emptyset$, then   $\mathfrak{T}_{\Gamma}=\emptyset$.  Define  $A\dslash\Gamma$  to be the triple $(E,*, \mathfrak{T}_{\Gamma})$. It is easily checked that in general, $A\dslash \Gamma$ is a preacycloid and
$A\dslash\Gamma=A\dslash \Gamma^{*}$. If $L$ is the set of loops of $A$, then  $\Gamma\cup \Gamma^{*}\cup L$ is contained in the set of  loops of $A\dslash \Gamma$.
We will  use the construction $A\dslash \Gamma$ in this section  only for parallelism classes $\Gamma$, but use it more generally in Section 4.

Following \cite{topechar} but taking into account Remark \ref{signedset}, we  define the \emph{elementary contraction} $A/e$ of  the preacycloid  $A=(E,*,\mathfrak{T})$  at $e\in E$ as follows. Let $E':=E\setminus \{e,e^{*}\}$ and let $\dag$ denote the restriction of $*$ to an involution on $E'$.  If  $e$ is  a loop of $E$, define $A/e:=(E',\dag,\mathfrak{T})$, which is obviously a preacycloid. If $e$ is not a loop, define $ A/e:=(E',\dag,\mathfrak{T}_{[e]})$ where $[e]$ is the parallelism class of $e$; this is a preacycloid too since  $A/e=(A\dslash[e])/e$ where $e$ is a loop of  $A\dslash[e]$.  (This   elementary contraction  corresponds to the contraction of an oriented matroid by the  set $\{e,e^{*}\}$, when the acycloid is the tope acycloid of an oriented matroid, though we shall not need this.)

 Let $A$, $B$ be preacycloids.
Say that $B$ is an \emph{elementary contraction}   of $A$ if  $B=A/e$ for some  element (loop or non-loop)  $e$ of $A$.   Call $B$  an \emph{elementary quasicontraction}   of $A$ if it is equal to $A\dslash[e]$ for some  non-loop $e$ of $A$.
We say that $B$ is a \emph{contraction} (respectively, \emph{quasicontraction}) of  $A$ if there is $n\in \mathbb{N}$ and  a sequence $A=A_{0},A_{1},\ldots, A_{n}=B$ of preacycloids such that for each $i=1,\ldots, n$, $A_{i}$ is an elementary contraction (respectively, elementary quasicontraction) of $A_{i-1}$.

The next theorem states a  key fact for the proof of the main results in this paper, namely Handa's characterization of oriented matroids in terms of acycloids and their contractions,   along with  its  trivial  reformulation (replacing contractions by quasicontractions) which is better adapted for purposes here.
\begin{theorem}\label{handa}
Let $A=(E,*,\mathfrak{T})$ be a preacycloid.    Then the following are equivalent:

 (i) there is some
 oriented matroid $M=(E,*,\cx)$ whose tope (pre)acycloid  is $\mathfrak{T}$.

 (ii) every contraction of $A$ is an acycloid.

 (iii) every quasicontraction of $A$  is an acycloid.
\end{theorem}
\begin{proof} The equivalence of (i) and (ii)  is from  \cite{topechar}.   (The fact that (i) implies (ii) was previously known from descriptions of  contractions of  oriented matroids in terms of their topes, and that (ii) implies (i)  was  conjectured by Tomizawa.)

We sketch a proof of the equivalence of  (ii) and (iii)  using the following    observations on ``triviality of loops'' (the proofs of which we omit).

Let $B$ be a  preacycloid, $f$ be a loop of $B$, and $e$ be an element  of $B$. Then

(a)  $B$ is an acycloid if and only if $B/f$ is an acycloid.

(b) If $e\neq f,f^{*}$, then $e$ is a non-loop of $B$ if and only if it is a non-loop of  $B/f$. In that case, the parallelism classes of $e$ in $B$ and $B/f$ are equal (that is,  $[e]_{B}=[e]_{B/f}$).

(c) If $e\neq f,f^{*}$ and $e$ is not a loop of $B$, then $B\dslash [e]/f=B/f\dslash [e]$ where
$[e]:=[e]_{B}=[e]_{B/f}$.

(d) If $e\neq f,f^{*}$, then $B/e/f=B/f/e$.

Above and  below,  the omitted parentheses should be left justified;
for example,  $B\dslash [e]/f:=(B\dslash [e])/f$.  Using these facts,
one sees by induction on $n$ that
\begin{equation*}
A/ e_{1}/\ldots /e_{n} =
A\dslash [e_{i_{1}}]\dslash\ldots\dslash [e_{i_{p}}]/e_{1}/\ldots /e_{n}
\end{equation*} if the left hand side is defined (that is, if $e_{j}\neq  e_{i}, e_{i}^{*}$ for $i<j$),  where $i_{1}<\ldots <i_{p}$ are the indices $i$ such that the elementary contraction $/e_{i}$  on the left hand side is at a non-loop $e_{i}$ (of $A/e_{1}/\ldots /e_{i-1}$).
Here, the parallelism class $[e_{i_{j}}]$ is taken in $A\dslash [e_{i_{1}}]\dslash\ldots\dslash [e_{i_{j-1}}]$, for $j=1,\ldots, p$.
 The elementary contraction $/e_{j}$ on the right hand side  is at a loop (of $A\dslash [e_{i_{1}}]\dslash\ldots\dslash [e_{i_{p}}]/e_{1}/\ldots /e_{j-1}$), for $j=1,\ldots, n$.
The identity applies in particular if $A\dslash{[e_{1}]}\dslash \ldots \dslash [e_{n}]$ is defined (that is, $e_{j}$ is not a loop in $A\dslash{[e_{1}]}\dslash \ldots \dslash [e_{j-1}]$ for $j=1,\ldots, n$), with $i_{j}=j$ for $j=1,\ldots, n$. It follows that the set of preacycloids arising as a  contraction of $A$ is the same as the set of  preacycloids obtained by applying  a (possibly empty) sequence of successive contractions at loops to some quasicontraction of $A$.  Then $\text{\rm (ii)$\iff$(iii)}$ follows from (a)-(d).
\end{proof}

We leave the interested reader to check that, for  a preacycloid $A$ as in the theorem, $A$ is not the tope (pre)acycloid of any oriented matroid (that is, the condition \ref{handa}(i) fails) if and only if there a contraction (respectively, quasicontraction) of $A$ which has no topes.

\subsection{Groupoids with root systems}\label{groupoidsetnotions}
 In the remainder of this  section we describe rudimentary properties of  the notion of  signed groupoid set as defined in \cite{rootoid1}.  We also discuss some  additional conditions one  may  impose which give rise to classes of structures which abstract, with varying degrees of generality, certain basic features of Coxeter groups and their root systems.

A \emph{groupoid} is a small category in which every morphism has an inverse.
A groupoid is called \emph{connected} if  it has at least one object and  for any     of its  objects $a,b$, there is at least one
 morphism from $b$ to $a$.
A groupoid is called \emph{simply connected} if for any  of its  objects $a,b$, there is   at most one morphism from $b$ to $a$.
A groupoid is said to be \emph{finite} if the set of  its morphisms (and hence also the set of its objects)  is finite.
Let $G$ be a groupoid and $a,b$ be objects  of $G$. We denote   by   $\lsub{a}{G}$ the set of morphisms with codomain $a$. Denote  by   $\lrsub{a}{G}{b}$ the set of morphisms from $b$ to $a$.  Denote  by   $\Ob  (G)$ the set of objects of $G$ and  by  $\Mor   (G)$ the set of morphisms of $G$. A \emph{subgroupoid} $H$ of $G$ is a subcategory of $G$ such that for any morphism in $H$, its inverse in $G$ is also contained in $H$.   A   maximal connected subgroupoid of $G$ is called a \emph{connected  component} of $G$.   For  $a\in \Ob  (G)$, the connected component containing $a$ is denoted $G[a]$.  A subgroupoid $H$ is said to be  \emph{full} if it contains all morphisms in $G$ between any two objects of $G$.  We say a subgroupoid $H$ is a \emph{union of components of $G$} if  every component of $H$ is a component of $G$. We denote by $1_{a}$ the identity morphism at the object $a$.

 For an involuted set $(E,*)$, we often  write $-x:=x^*$.  A \emph{definitely involuted set} $E$ is a strictly involuted set $E$ together with a chosen subset $E^+\subseteq E$ such that $E=E^+\dotcup  (E^+)^*$  where we use $\dot\cup$ to denote a disjoint union. We denote $(E^+)^*$ by $-E$.

\subsection{} A \emph{signed groupoid set} is a triple $(G,\Phi,\Phi^+)$ where $G$ is a groupoid and   $\Phi=(\lsub{a}{\Phi})$ is a family of definitely involuted sets indexed by the objects of $G$ such that  $G$  acts on $\Phi$, via maps   $\lrsub{a}{G}{b}\times \lsub{b}{\Phi}\rightarrow \lsub{a}{\Phi}$  for $a,b\in \Ob(G)$,  with the properties:

(i) $1_b(x)=x$ where $x\in \lsub{b}{\Phi}$,

(ii) $f(g(x))=(fg)(x)$ where $x\in \lsub{b}{\Phi}$, $g\in \lrsub{a}{G}{b}$ and  $f\in \lsub{c}{G}_a$,

(iii) $f(-x)=-f(x)$ where $x\in \lsub{b}{\Phi}$ and $ f\in \lrsub{a}{G}{b}$.

We will call the elements in $\lsub{a}{\Phi}$ \emph{roots} (at object $a$) and the elements in $\lsub{a}{\Phi}^+$ (respectively $\lsub{a}{\Phi}^-$) \emph{positive roots} (respectively \emph{negative roots}) at object $a$. The collection of   definitely involuted sets $\lsub{a}{\Phi}$  is called the \emph{root system} of $G$. For convenience, sometimes we write $\alpha>0$ (respectively $\alpha<0$) if $\alpha$ is a positive root (respectively if $\alpha$ is a negative root).

 We   shall not define a category of signed groupoid sets  in this paper, but will occasionally   use  the obvious notion of an isomorphism of signed groupoid sets.

 For a signed groupoid set $R=(G,\Phi,\Phi^+)$ and a subgroupoid $H$ of $G$, let the restriction
$R_H$ of $R$ to $H$  be the signed groupoid set $R_H:=(H,\Psi,\Psi^+)$ defined as follows:
for any $a\in \Ob(H)$, $\lsub{a}{\Psi}:=\lsub{a}{\Phi}$ as definitely involuted set, and for any $a,b\in \Ob(H)$, the
map $\lrsub{a}{H}{b}\times \lsub{b}{\Psi}\to \lsub{a}{\Psi}$ is the restriction of the map $\lrsub{a}{G}{b}\times \lsub{b}{\Phi}\rightarrow \lsub{a}{\Phi}$. If $H$ is a component (respectively, a union of components) of $G$, we say that $R_H$ is a component (respectively, a union of components) of $R$.
For an object $a$ of $R$ (that is, an object $a$ of the underlying groupoid $G$ of $R$), we let $R[a]$ denote the component of $R$ whose underlying groupoid is $G[a]$.

\begin{example}\label{coxeterexample}(\emph{Coxeter groups and real reflection groups})
 A lot of the definitions and terminology we use for signed groupoid sets is motivated by standard notions in the study of Coxeter groups and reflection groups, but the setting is   far more general and familiarity with Coxeter groups and root systems is not required in the proofs.
 We provide here some informal remarks for readers who are not familiar with these matters.

Consider a real vector space $V$ equipped with a symmetric bilinear (but not necessarily positive definite) form
$(-\mid -)\colon V\times V\to \mathbb{R}$.
A vector $\alpha\in V$ is said to be non-isotropic if $(\alpha\mid \alpha)\neq 0$. In that case, the corresponding reflection is defined to be the (unique) invertible linear map  $s_{\alpha}\colon V\to V$ which fixes the hyperplane orthogonal to $\alpha$ pointwise and maps $\alpha$ to $-\alpha$.  A  real reflection group on $V$ is a group $W$ of invertible linear maps of $V$ which is generated by a set of reflections.  By a root system for $W$ is generally meant  a $W$-stable set $\Phi$ of $V\setminus \{0\}$
such that $W$ is generated by reflections in vectors in some subset  of $\Phi$ (this subset is often but not always equal to $\Phi$, depending on the context in which $W$ and $\Phi$ arise).

Coxeter groups $W$ are a class of groups which are defined by existence of a certain simple type of presentation with a  special  set of generators  $S$, the elements of which are called  ``simple reflections.''  One then calls $(W,S)$ a Coxeter system. Coxeter groups  always have faithful representations as reflection groups with root systems as above.     The finite Coxeter groups arise by taking $(V,(-\mid -))$ to be an inner product space (i.e. $V$ is finite dimensional and  the form is positive definite) and $W$ as a finite group generated by reflections on $V$.  For the root system, one may take the set of all unit vectors $\alpha\in V$  such that $s_{\alpha}$ in $W$, but other choices are often more natural in applications (for instance,  ``crystallographic''  root systems of finite Weyl groups arise naturally in the structure theory of semisimple complex Lie algebras).

Coxeter groups in general
have a similar  ``standard geometric representation'' defined from their presentation,  as described in \cite{bjornerbrenti} or \cite{Hum}, and a standard  root system $\Phi$ such that $s_{\alpha}\in W$ for  all $\alpha\in \Phi$. Certain (crystallographic) Coxeter groups also arise naturally as   Weyl groups of  Kac-Moody Lie algebras (\cite{Kac}); in this case, the root system is the disjoint union of a set of ``real roots''
(the reflections in which generate $W$) and a set of  ``imaginary roots''
(roots in which  may even be isotropic and so have no corresponding reflection).

The root system, as a subset of $V\setminus\{0\}$ stable under multiplication by $\pm 1$, may be regarded as a realizable oriented matroid. For Coxeter groups, its  construction or properties gives rise to a standard (up to $W$-action) hemispace of $\Phi$, called the standard positive system $\Phi^{+}$, elements of which are called positive roots. Typically,  $S$  may be  characterized geometrically either  as  the  reflections $s_\alpha$ in the roots $\alpha$   which span the extreme rays  of  $\cone(\Phi^{+})$, or  combinatorially as the group elements which map only one positive  root  (and its other positive scalar multiples, if any) outside  $\Phi^{+}$.
For a reflection group $W$ which is not a Coxeter group, there is usually  no  canonical choice of positive system, but  one may  define  $\Phi^{+}$ to be an arbitrarily chosen hemispace of  $\Phi$; simple reflections in either sense above do  not  then typically generate $W$.
For example, the  orthogonal group  of  a real Euclidean space,   with the root system being the unit sphere,  is examined from this point of view in  \cite{orthogonalgroup}.

 Let $\Phi$ be a root system of a Coxeter group  (or more generally, real reflection group)  $W$ as above.  Regard the group $W$ as a groupoid $G$ with a single object $\bullet$,  and the set of morphisms being the set of elements in $W$. Let $\lsub{\bullet}{\Phi}=\Phi, \lsub{\bullet}{\Phi}^+=\Phi^+$. The action of morphisms on the roots is exactly the action of the
 corresponding elements in $W$ on the roots. Then $(G,\lsub{\bullet}{\Phi},\lsub{\bullet}{\Phi}^+)$ is a signed groupoid set.
\end{example}

\begin{example} (\emph{Preacycloids}) \label{exampleacycloid}
Let $A=(E,*,\mathfrak{T})$ be a preacycloid with $\mathcal{T}\neq \emptyset$ and  $L$ as its set of loops.
Choose a subset $L^+$ of $L$ such that $L=L^+\dotcup (L^+)^*$.
Consider a connected, simply connected groupoid $G$ whose objects are indexed by the topes of $A$. For $H\in \mathfrak{T}$, denote the corresponding object by $\widetilde{H}$. The set $\lsub{\widetilde{H}}{\Phi}$ of roots at an object $\widetilde{H}$  is $E$ and the set $\lsub{\widetilde{H}}{\Phi}^+$ of positive roots at this object  is $H\cup L^+$.
Any morphism acts on $E$ as identity. Then $R=(G,\Phi,\Phi^+)$ is a signed groupoid set. In particular, associated to an  oriented matroid, one can construct such a signed groupoid set
(by taking $A$ above as the tope  (pre)acycloid of  the  oriented matroid).
We shall  denote this signed groupoid set $R$ as $\SGS (A)$ to denote its dependence on $A$. It is easy to see that the isomorphism type  of $\SGS(A)$ is independent of the choice of $L^{+}$.
\end{example}

\begin{example} (\emph{Brink-Howlett Groupoid})\label{brinkhowexamp}
Let $(W,S)$ be a Coxeter group with standard root system $\Phi$,
$\Phi^+$  be   the standard positive system and $\Delta\subseteq \Phi^+$  denote its simple system (the positive roots corresponding to the simple reflections). Construct a groupoid with objects being the subsets of $\Delta$ and a morphism  from $J$ to $I$   being of the form $(I,w,J)$ where $I,J\subseteq \Delta$ and  $w\in W$ with  $w(I)=J$. At each object, let the root system and the set of positive roots  be inherited  from those of $W$; that is,  $\lsub{I}{\Phi}=\Phi$ and $\lsub{I}{\Phi}^+=\Phi^+$. Equipped with  such  root systems  at its objects , the groupoid becomes a signed groupoid set.
 This  groupoid (though not its root system as defined here)   is considered in \cite{bh} for the purpose of studying the normalizer of parabolic subgroups of $W$. We will later show that this groupoid set can be obtained by applying the
 generalized   Brink-Howlett  construction, described  in Section \ref{groupoidsetnotions}, to the signed groupoid set  in Example \ref{coxeterexample}.
(We remark that there are also more subtle choices of root system, which we do not discuss in this paper,  for the Brink-Howlett groupoid.)
\end{example}

\begin{example}\label{weylgroupoidexample} (\emph{Weyl groupoid})
Given a Cartan scheme $\mathcal{C}$, one can associate to it a Weyl groupoid $G$. Suppose that $G$ has a root system of type $\mathcal{C}$, in the sense of \cite{weylgroupoid5}.  At each object,
the root system is a (definitely involuted) subset of the free $\mathbb{Z}-$module of  rank equal to the rank of  $\mathcal{C}$. The  morphisms of the Weyl groupoid
can be considered as automorphisms of $\mathbb{Z}^{\text{rank}(\mathcal{C})}$ and the action of them on the roots satisfies the required
conditions of a signed groupoid set. For details of Weyl groupoids,  see  \cite{weylgroupoid}, \cite{weylgroupoid1}, \cite{weylgroupoid2}, \cite{weylgroupoid3}, \cite{weylgroupoid4}, \cite{weylgroupoid5} and  \cite{weylgroupoid6}.
\end{example}

\subsection{} We say that a signed groupoid set $(G,\Phi,\Phi^+)$ is \emph{finite} (respectively \emph{connected}, \emph{simply connected}) if  is $G$ is finite and $\lsub{a}{\Phi}$ is finite for all $a\in \Ob(G)$ (respectively $G$ is connected, simply connected). For $g\in \lrsub{a}{G}{b}$, we define the \emph{inversion set} \begin{equation*}
\Phi_g=\lsub{a}{\Phi}^+\cap g(\lsub{b}{\Phi}^-)=\{\alpha\in \lsub{a}{\Phi}^+\mid g^{-1}(\alpha)\in \lsub{b}{\Phi}^-\}.
\end{equation*}

A positive (respectively negative) root $\alpha$ in $_a\Phi^+$ (respectively $_a\Phi^-$) is called \emph{imaginary} if $\alpha\not\in \Phi_g$ for any
$g\in \lsub{a}{G}$ (respectively if $-\alpha\not\in \Phi_g$ for any $g\in \lsub{a}{G}$). A root that is not imaginary is called \emph{real}.
If for all $_a\Phi$,  where  $a\in \Ob  (G)$, the set of imaginary roots is empty then we call $(G,\Phi,\Phi^+)$ \emph{real}.

Denote the  set of positive (respectively, negative) imaginary roots of $\lsub{a}{\Phi}$ as $\tensor*[_{a}]{\Phi}{_{\mathrm{im}}^+}$ (respectively,  $\tensor*[_{a}]{\Phi}{_{\mathrm{im}}^-}=-\tensor*[_{a}]{\Phi}{_{\mathrm{im}}^+}$) and let  $\tensor*[_{a}]{\Phi}{_{\mathrm{im}}}
 :=\tensor*[_{a}]{\Phi}{_{\mathrm{im}}^+}\cup\tensor*[_{a}]{\Phi}{_{\mathrm{im}}^-}$.  Let $\lrsub{a}{\Phi}{\mathrm{re}}=\lsub{a}{\Phi}\setminus \tensor*[_{a}]{\Phi}{_{\mathrm{im}}}$,
   $\tensor*[_a]{\Phi}{_{\mathrm{re}}}=\tensor*[_a]{\Phi}{^{+}}\setminus \tensor*[_{a}]{\Phi}{_{\mathrm{im}}^+}$ and
    $\tensor*[_a]{\Phi}{_{\mathrm{re}}^{-}}=\tensor*[_a]{\Phi}{^{-}}\setminus \tensor*[_{a}]{\Phi}{_{\mathrm{im}}^-}$ denote the sets of all, all positive and all negative real roots at $a$, respectively.
    Observe that for any morphism
    $g\colon b\to a$ in $G$, we have  $g(\tensor*[_{b}]{\Phi}{_{\mathrm{im}}^{\pm}})=\tensor*[_{a}]{\Phi}{_{\mathrm{im}}^{\pm}}$,   $g(\tensor*[_{b}]{\Phi}{_{\mathrm{re}}})=\tensor*[_{a}]{\Phi}{_{\mathrm{re}}}$ and  $\Phi_{g}=\tensor*[_{a}]{\Phi}{_{\mathrm{re}}^+}\cap g(\tensor*[_{b}]{\Phi}{_{\mathrm{re}}^-})$.

  The following  facts are from \cite{rootoid1}, where they are expressed in the language of protorootoids and checked by routine computations with $1$-cocycles.
\begin{lemma}\label{sgsfacts1}
Let $R=(G,\Phi,\Phi^+)$ be a signed groupoid set.

(a) Let $f$ and $g$  be   morphisms in $G$  such that the composite $fg$ is defined. Then we have  $\Phi_{fg}=(\Phi_f\setminus -f\Phi_g)\dotcup  (f\Phi_g\setminus -\Phi_f)$.

(b) For any morphism $h$ of $G$, we  have $\Phi_{h^{-1}}=-h^{-1}{\Phi_h}$.

(c) Let $f,g$ be two composable morphisms in $G$. We have the equivalence: $\Phi_{f^{-1}}\cap \Phi_g=\emptyset \Leftrightarrow \Phi_{f}\subseteq \Phi_{fg} \Leftrightarrow \Phi_{fg}=\Phi_f\dotcup  f\Phi_g$.

(d) For morphisms $h,g\in \lsub{a}{G}$, we have $\Phi_h=\Phi_g$ if and only if $\Phi_{h^{-1}g}=\emptyset$.
\end{lemma}
\begin{proof}
(a) By definition \begin{equation*}
\Phi_{fg}=\{\alpha\mid \alpha\in \Phi_f, -f^{-1}(\alpha)\not\in \Phi_g\}\dotcup  \{\alpha\mid \alpha>0, f^{-1}(\alpha)\in \Phi_g\}.
\end{equation*}
Note that  $\alpha<0$ and $f^{-1}(\alpha)\in \Phi_g$ imply $-\alpha\in \Phi_f$. Therefore the equation holds.

(b) This follows from (a) on taking $f:=h^{-1}$ and $g:=h$.

(c) Suppose $\Phi_{f^{-1}}\cap \Phi_g=\emptyset$.  Note that it follows from  (b) that $\Phi_{f^{-1}}=-f^{-1}\Phi_f$.  Therefore $\Phi_f\cap -f\Phi_g=\emptyset$.  By (a) $\Phi_f\subseteq \Phi_{fg}$. Suppose that $\Phi_f\subseteq \Phi_{fg}$. Again by (a) this implies that $\Phi_f\cap -f\Phi_g=\emptyset$ (and therefore $f\Phi_g\cap -\Phi_f=\emptyset$). Use (a) again we see that $\Phi_{fg}=\Phi_f\dotcup  f\Phi_g$.  Finally suppose that $\Phi_{fg}=\Phi_f\dotcup  f\Phi_g$.  By (a) this implies that $\Phi_f\cap -f\Phi_g=\emptyset$. Hence $-f^{-1}\Phi_f\cap f^{-1}f\Phi_g=\emptyset$. Therefore $\Phi_{f^{-1}}\cap \Phi_g=\emptyset$.

(d) This follows by taking $f:=h^{-1}$  in (a) and using (b).  \end{proof}

\subsection{}\label{sgsterm}  Let $R=(G,\Phi,\Phi^{+})$ be a signed groupoid set.  If $\Phi_g=\emptyset$ implies that $g$ is the identity morphism then we say $(G,\Phi,\Phi^+)$ is \emph{faithful}.  By Lemma \ref{sgsfacts1}, this holds if and only if for any $a\in \Ob(G)$ and any  $g,h\in \lsub{a}{G}$ with $\Phi_g=\Phi_h$, one has $ g=h$.

Assume for the rest of this subsection that $R$ is faithful, unless otherwise stated\footnote{Here and elsewhere in this paper, assumptions of faithfulness can sometimes be  removed, though often at the expense of more cumbersome statements or notation.}.

There is a partial order $\leq_a$ on $\lsub{a}{G}$, called the \emph{weak order} of $G$ (at $a$), such that $g\leq_{a} h$ if and only if $\Phi_g\subseteq \Phi_h$. Note that $1_a
$ is the minimum element of $(\lsub{a}{G},\leq_a)$.    In what follows  when we compare two morphisms at a given object, it is always understood that they are compared in the sense of weak order.
  If $g,h\in \lsub{a}{G}$, we may write $g\leq h$ instead of $g\leq_a
h$ if confusion is unlikely.

A morphism $g$ of $G$ is said to be \emph{simple} if $\vert \Phi_g\vert =1$.
 By Lemma \ref{sgsfacts1}(b), the inverse of a simple morphism is simple.   The set of simple morphisms with codomain $a$ is denoted $\lsub{a}{S}$.
We call a morphism $g\in \lsub{a}{G}$ \emph{atomic} if   $g$ is an atom in the poset $\lsub{a}{G},\leq_a)$ (that is, if $h\in \lsub{a}{G}$ with $h\leq_a g$ implies
  $h=g$ or $h=1_a$).
Clearly a simple morphism is atomic.

 We  call   $R$ \emph{interval finite} if for any object $a$ and morphism $g\in \lsub{a}{G}$, the set  $\{h\mid h\in \lsub{a}{G}, \Phi_h\subseteq \Phi_g\}$  is finite  (that is, the closed interval $[1_a,g]$ in the  weak order $\leq_a$ on $\lsub{a}{G}$ is finite).
  We say  $(G,\Phi,\Phi^+)$ is  \emph{inversion-set finite} if for any $g\in \Mor   (G)$, $\Phi_g$ is finite. An inversion-set finite signed groupoid set is clearly interval finite.

  Say that  the groupoid $G$ is \emph{generated} by a set $X\subseteq \Mor(G)$ if every non-identity morphism in $G$ is expressible as a composite of morphisms in $X\cup X^{-1}$, where $X^{-1} :=\{g^{-1}\mid g\in G\}$. For $g \in \Mor(G)$,  define the \emph{length} $l_X(g)\in \mathbb{N}$ of $g$ (with respect to $X$) to be $0$ if $g$ is an identity morphism, and otherwise to be the minimum number of factors occurring  in   products of elements of $X\cup X^{-1}$ with value $g$.

  A (not necessarily faithful)   signed groupoid set $R=(G,\Phi,\Phi^+)$  is called \emph{principal} if  the set $S:=\bigcup_{a\in \Ob  (G)}\lsub{a}{S}$  of simple morphisms  generates  $G$ and for all $g\in \Mor(G)$,
  $l(g):=l_S(g)$ satisfies $l(g)=\vert \Phi_g\vert $.  A principal  signed groupoid set is necessarily faithful, since  if $g\in \Mor(G)$ satisifes  $ \Phi_g =\emptyset$,
  then $l(g)=\vert \Phi_g\vert=0$ and so $g$ is an identity morphism.

 If  $G$ is generated by  its  atomic morphisms, we say that $R$ is \emph{atomically generated}.  An interval finite,  faithful  signed groupoid set is called \emph{preprincipal} if for any $g,s\in \lsub{a}{G}$, with $s$ being atomic, one has either $\Phi_g\supseteq \Phi_s$ or $\Phi_g\cap \Phi_s=\emptyset$.

 We say that $R$ is \emph{antipodal} if for each $a\in \Ob(G)$, the weak order $(\lsub{a}{G},\leq_{a})$ has  a maximum element (this notion is the only one from above that  is not already considered  in \cite{rootoid1}--\cite{rootoid2}). In general, a maximum element of $(\lsub{a}{G},\leq_{a})$ will be denoted $\omega_{a}\colon w_{a}\to a$ if it exists.

For completeness, the following lemma states and proves, using the terminology of signed groupoid sets,  some additional facts formulated in  terms  of protorootoids in \cite{rootoid1}.

\begin{lemma}\label{sgsfacts}
Let $R=(G,\Phi,\Phi^+)$ be a  faithful  signed groupoid set.

(a) Let $x\in \lrsub{a}{G}{b}$,   $y\in \lrsub{b}{G}{c}$
  and $w\in \lsub{b}{G}$.
  \begin{itemize}
 \item If $x\leq_a xy$, then $y^{-1}\leq_c y^{-1}x^{-1}$.  \item If $x<_a xy$ and $x<_axw$, then $xy<_axw$ if and only if $y<_bw$.
  \end{itemize}

 (b) The inverse of an atomic morphism  is atomic.

(c) Suppose that  $G$ is  generated by  its  simple morphisms. Then $\vert \Phi_g\vert \leq l(g)$  for all $g\in \Mor(G)$   and thus $R$ is inversion-set finite.

(d) Suppose that  $R$ is  interval finite. Then $R$ is atomically generated.

(e) Let $R$ be interval finite. Suppose  that any atomic morphism of $R$ is simple. Then $R$ is principal.

(f) If $R$ is principal, then it is preprincipal  and every  atomic morphism of $R$ is simple.
\end{lemma}
\begin{proof}

(a)
 The  two assertions  follow directly from  Lemma \ref{sgsfacts1}(c)  and (b).

 (b) It follows from  the first assertion of  (a).

(c) We prove this by induction (the argument does not require the assumption that $R$ is  faithful).  If $l(g)=0$,  then  $g$ is the identity morphism at some object and thus $\vert \Phi_g\vert =0$.
Suppose  $\vert \Phi_g \vert\leq l(g)$ for $g$ such that $l(g)<n$.  Now assume that $l(g)=n$.  Then $g=s_ns_{n-1}\cdots s_1$ where each $s_i$ is simple. By Lemma \ref{sgsfacts1} (a), \begin{equation*}
\Phi_g=(\Phi_{s_n}\setminus -s_n\Phi_{s_{n-1}\cdots s_1})\,\dotcup  \,(s_n\Phi_{s_{n-1}\cdots s_1}\setminus -\Phi_{s_n}).
\end{equation*} By definition one sees that $l(s_{n-1}\cdots s_1)=n-1$,
Therefore by induction, we have $\vert \Phi_{s_{n-1}\cdots s_1}\vert \leq n-1$.  Note $\vert \Phi_{s_n}\vert =1$. Therefore $\vert \Phi_g\vert \leq n$.

(d) Take a morphism $g\in \lsub{a}{G}$. Denote the cardinality   of the interval $[1_a, g]$ in the weak order at $a$ by $l'(g)$.   We use induction on $l'(g)$ to show that   if $g$ is not an identity morphism, then $g$ is a product of  atomic morphisms.
If $l'(g)=1$ then $\Phi_g=\emptyset$ and $g=1_a$ by faithfulness. Suppose that $g$  is a product of  atomic morphisms if $1<l'(g)<n$. Suppose $l'(g)=n$.  Take an atom  $r\in \lsub{a}{G}$  such that $\Phi_r\subseteq \Phi_g$.   Lemma \ref{sgsfacts1} (c) implies that $\Phi_{r^{-1}}\cap \Phi_{r^{-1}g}=\emptyset$ and  $\Phi_{g}=\Phi_r\dotcup  r\Phi_{r^{-1}g}$.  Let $b$ denote the domain of $r$.  We will show that $l'(r^{-1}g)<l'(g)$ and then the assertion follows from the induction.   To that end,   we show that there exists a bijection between the interval $[1_b, r^{-1}g]$ and the interval $[r,g]$ under the map $h\mapsto rh$.

Since $\Phi_h\subseteq \Phi_{r^{-1}g}$ and $\Phi_{r^{-1}}\cap \Phi_{r^{-1}g}=\emptyset$,  we have   $\Phi_{rh}=\Phi_r\dotcup  r\Phi_h$ by Lemma \ref{sgsfacts1} (c). Hence $\Phi_r\subseteq \Phi_{rh}\subseteq \Phi_{g}=\Phi_r\, \dotcup \, r\Phi_{r^{-1}g}$. Therefore the map is well-defined. It follows from the invertibility of $r$ that the map is injective.
Take $h'\in \lsub{a}{G}$ such that $r\leq_a h'\leq_a g$. Again by Lemma \ref{sgsfacts1}  (c), $\Phi_{h'}=\Phi_r\, \dotcup \, r\Phi_{r^{-1}h'}$. It follows  that $\Phi_{r^{-1}h'}\subseteq \Phi_{r^{-1}g}$. Therefore one sees that the map is surjective. Hence $l'(r^{-1}g)=\vert [r,g]\vert <\vert [1_{a},g]\vert =l'(g)$.

(e) By (d) and the assumption  any  non-identity  morphism of $R$  is a product of  simple morphisms. We have to show that for any morphism $g\in \lsub{a}{G}$, $\vert \Phi_g\vert =l(g)$. We prove this by induction on $\vert \Phi_g\vert $. If $\vert \Phi_g\vert =0$, then $\Phi_g=\emptyset$ and $g$ is  an  identity morphism since $R$ is faithful.
 Therefore $l(g)=0$.   Now assume   inductively   that if $\vert \Phi_g\vert <n$, then  $l(g)=\vert \Phi_g\vert $.
 Consider   $g$ such that $\vert \Phi_g\vert =n>0$. Because $R$ is interval finite, we can take an atomic morphism $r\in \lsub{a}{G}$ such that $r\leq_a g$. By assumption $r$ is simple (and therefore $r^{-1}$ is also simple). We have $\Phi_{g}=\Phi_r\dotcup  r\Phi_{r^{-1}g}$ by Lemma \ref{sgsfacts1}(c)  and therefore $\vert \Phi_{r^{-1}g}\vert =\vert \Phi_g\vert -1$. By induction $\vert \Phi_{r^{-1}g}\vert =l(r^{-1}g)=n-1$.  By the definition of $l$, $l(g)\leq n$. On the other hand, by (c), $n\leq l(g)$.  Therefore the assertion follows.

(f) Let $R$ be a principal signed groupoid set. By (c), a principal signed groupoid set is inversion-set finite, and thus interval finite. We now show that every atomic morphism of $(G,\Phi, \Phi^+)$ is simple. Take $r$ an atomic morphism. If it is not simple, suppose $l(r)=k>1$ and $r=s_ks_{k-1}\cdots s_1$ with $s_i$ simple. Since $R$ is principal, $\vert \Phi_r\vert =k$.  By Lemma \ref{sgsfacts1} (a), \begin{equation*}
\Phi_r=(\Phi_{s_k}\setminus -s_k\Phi_{s_{k-1}\cdots s_1})\dotcup  (s_k\Phi_{s_{k-1}\cdots s_1}\setminus -\Phi_{s_i})
\end{equation*} By definition $l(s_{k-1}\cdots s_1)=k-1$. By (b)
$\vert \Phi_{s_{k-1}\cdots s_1}\vert \leq k-1$.  Therefore this forces $\Phi_{r}=\Phi_{s_k}\dotcup  s_k\Phi_{s_{k-1}\cdots s_1}$. Therefore $\Phi_{s_k}\subseteq \Phi_r$ by Lemma \ref{sgsfacts1}  (c),  a contradiction. Therefore for an atomic morphism $r$ and a morphism $g$ having the same codomain as $r$, either $\Phi_r\subseteq \Phi_g$ or $\Phi_g\cap \Phi_r=\emptyset$, since $\Phi_r$ is a singleton set.
\end{proof}

 \begin{lemma} \label{parinpreprinc}Let $R=(G,\Phi,\Phi^{+})$ be a faithful signed groupoid set,  $X$ be a set of generators of $G$ such that $\{g^{-1}\mid g\in X\}=X$, $a\in \Ob(G)$ and $\alpha\in \lrsub{a}{\Phi}{\mathrm{re}}$. Then there exist some  $y\in \lrsub{b}{G}{a}$ and  $s\in X\cap \lsub{b}{G}$ such that
$y(\alpha)\in \Phi_{s}$. In particular, this applies with $R$
atomically generated  (respectively, preprincipal,  principal) and $X$ equal  to  the set of atomic (respectively, atomic,  simple) morphisms  of $R$.
\end{lemma}
\begin{proof} Since $\alpha$ is a real root, there is some morphism $g$ such that $\alpha$ and $g(\alpha)$ are of opposite sign. Obviously,  $g$ is not an identity morphism, so it is a product $g=s_{n}\cdots s_{1}$ of  morphisms $s_{i}$ in the set $X$. For some $i=1,\ldots,n$, $s_{i}\cdots s_{1}(\alpha)$ and $s_{i-1}\cdots s_{1}(\alpha)$ are of opposite sign.
Hence there is $x\in G$ and a morphism $r$ in $X$ such that
$x(\alpha)$ and $rx(\alpha)$ are of opposite sign.
If $x(\alpha)>0$ and $rx(\alpha)<0$, then $x(\alpha)\in \Phi_{r^{-1}}$ where $r^{-1}$ is in $X$; in this case, we can choose $y=x$ and $s=r^{-1}$. Otherwise, we have
$x(\alpha)=r^{-1}(rx(\alpha))<0$ and $rx(\alpha)>0$, so  $rx(\alpha)\in \Phi_{r}$; in this case, we can take $y=rx$ and $s=r$.
\end{proof}

\subsection{Real compression}   Let $R=(G,\Phi,\Phi_+)$ be any signed groupoid set. Define a preorder (that is, a reflexive, transitive relation) $\preceq_a$ on $\lsub{a}{\Phi}$ by the condition that for $\alpha$ and $\beta$ in $\lsub{a}{\Phi} $, one has $\alpha\preceq_a \beta$ if for all $b\in \Ob(G)$ and $g\in \lrsub{a}{G}{b}$, $g(\beta)\in \lsub{b}{\Phi}^-$ implies  $g(\alpha)\in \lsub{b}{\Phi}^-$. Following \cite{bhaut} (see also \cite{bjornerbrenti}), we call $\preceq_a$ the \emph{dominance preorder} at $a$.
Let $\sim_a$ denote the corresponding equivalence relation on $\lsub{a}{\Phi}$, defined by \begin{equation*}
\alpha\sim_a \beta\iff (\alpha\preceq_a \beta \text{ \rm and } \beta\preceq_a
\alpha),\qquad\text{\rm for $\alpha,\beta\in \lsub{a}{\Phi}$}.
\end{equation*} We call the relation $\sim_a$ \emph{parallelism} on $\lsub{a}{\Phi}$. Thus,
 roots $\alpha,\beta\in \lsub{a}{\Phi}$ are said to be \emph{parallel} if for all $g\in \lsub{a}{G}$, $g^{-1} (\alpha)$ and $g^{-1} (\beta)$ are of the same sign. Note that each of $\tensor*[_{a}]{\Phi}{_{\mathrm{im}}^\pm}$, if non-empty, is a single parallelism class.
 If $R$ contains no  distinct   parallel roots at any object, we call it \emph{compressed}.  Thus, $R$ is compressed if and only if  its dominance preorder at each object is a partial order (called \emph{dominance order}).

 We attach to $R$ another signed groupoid set $R^{\mathrm{rec}}:=(G, \Phi_\mathrm{rec},\Phi_\mathrm{rec}^+)$, called the \emph{real compresssion} of $R$, as follows.
 For $a\in \Ob(G)$, define the definitely involuted set $\lrsub{a}{\Phi}{\mathrm{rec}}$ by   \begin{equation*}
 \lrsub{a}{\Phi}{\mathrm{rec}}:=\lrsub{a}{\Phi}{\mathrm{re}}/\negthinspace \sim_a,\quad -[\alpha]_a:=[-\alpha]_a
\qquad \text{\rm for
 $\alpha\in \lrsub{a}{\Phi}{\mathrm{re}}$}
 \end{equation*}  where $[\alpha]_a$ is the $\sim_a$-equivalence class of $\alpha$, and
 $\tensor*[_{a}]{\Phi}{_{\mathrm{rec}}^+}:=\tensor*[_a]{\Phi}{_{\mathrm{re}}^+}/\negthinspace \sim_a$. The maps
 $\lrsub{b}{G}{a}\times \lrsub{a}{\Phi}{\mathrm{rec}}\to \lrsub{b}{\Phi}{\mathrm{rec}}$ defining the action of $G$ are given by $g[\alpha]_a:=[g\alpha]_b$. It is easy to check that this gives a well-defined signed groupoid set $R^{\mathrm{rec}}$.

 \begin{lemma}\label{preandprincipal}
Let $R$ be a faithful signed groupoids set.
Assume  that $R$ is preprincipal, with $A$ as its set of  atomic morphisms. Then

 (a)  the  parallelism classes of real roots in  $\lsub{a}{\Phi}$ are the sets $x(\Phi_{s})$ where $x\in \lrsub{a}{G}{b}$ and $s\in \lsub{b}{G}\cap A$.

 (b)  $R^{\mathrm{rec}}$ is principal.

(c) if $g$  is  in $\lsub{a}{G}$, then $l_{A}(g) =\vert \Phi_{g}/\negthinspace \sim_{a}\vert$ (the number of parallelism classes  in $\Phi_{g}$).
\end{lemma}
\begin{proof}

(a) By definition of preprincipal signed groupoid sets, the sets $\Phi_{s}$ for $s\in A$ are parallelism classes, and hence so are the sets $x(\Phi_{s})$.
Every real root appears in such a parallelism class by Lemma \ref{parinpreprinc}, and (a) follows.

(b) It is easy to see that $g$ is an atomic morphism of $R$ if and only if $g$ is  an atomic morphism of $R^{\mathrm{rec}}$. Now we show that an atomic morphism $g$ is also simple in $R^{\mathrm{rec}}$, i.e. $\vert \Phi_g\vert =1$ in $R^{\mathrm{rec}}$. Note that $R$ is atomically generated by Lemma \ref{sgsfacts} (d). Therefore $R^{\mathrm{rec}}$ is also atomically generated.   Atomic morphisms are simple in $R^{\mathrm{rec}}$ by (a) applied to $R^{\mathrm{rec}}$.  The result follows from Lemma \ref{sgsfacts} (e).

(c) If $R$ is principal, this follows from the definition of   principalness
since the parallelism classes  of the real roots are singletons  by (a) and Lemma \ref{sgsfacts}(f). It follows for faithful, preprincipal signed groupoid sets since both the length of a groupoid element with respect to the set of atomic generators and  the number of parallelism classes in an inversion set are invariant under real compression.
\end{proof}

\subsection{}  In this subsection, $R=(G,\Phi,\Phi^+)$ denotes a faithful signed groupoid set.   If for all $a\in \Ob  (G)$ ,  the  weak  order  $(\lsub{a}{G},\leq_{a})$ at $a$   is a complete lattice, we say that $R$ is \emph{complete}.

For two morphism $g,h\in \lsub{a}{G}$, if $\Phi_g\cap \Phi_h=\emptyset$,  we   write  $g\perp h$ and say they are \emph{orthogonal}.
 By Lemma \ref{sgsfacts1}(c), orthogonality is expressible in terms of the family of weak orders of $R$ at the objects of $G$.

We say $R$ is \emph{rootoidal} if  for any $a\in \Ob  (G)$,
the weak order $(\lsub{a}{G},\leq_{a})$ is a complete meet semilattice (that is, any of its  non-empty subsets has a meet (greatest lower bound))  and the weak orders satisfy  the following \emph{Join Orthogonality Property (JOP)}:  if $ h, g_i\in \lsub{a}{G}$,  where $ i\in I$, with ${g_i}\perp h$  for all $i$  and  the join (least upper bound) $g=\bigvee_{i\in I}g_{i}$ exists    in weak order at $a$, then $g\perp h$.  (We remark that a subset of a complete meet semilattice has a join if and only if it is bounded above; its  join is then the  meet of the set of  upper bounds of the subset.)

The condition that a signed groupoid set be rootoidal is crucial in extending many basic   facts which hold  for complete signed groupoid sets to non-complete ones. The main reason for mentioning it in this paper (where  our main results concern  complete signed groupoid sets anyway) is to make explicit the fact, which we shall  use several times,  that weak orders in complete signed groupoid sets have the JOP, as (e) of  the following lemma shows.

\begin{lemma}\label{JOPlem} Let $R=(G,\Phi,\Phi^{+})$ denote a faithful signed groupod set.

(a) If $a\in \Ob(G)$, an element $\omega_{a}\in \lsub{a}{G}$ is a maximum element of weak order at $a$ if and only if $\Phi_{\omega_{a}}=\tensor*[_{a}]{\Phi}{_{\mathrm{re}}^{+}}$.

In (b)--(c), we  assume that $R$ is antipodal  and  for each $a\in \Ob(G)$,  let
$\omega_{a}\colon w_{a}\to a$  denote the maximum element in weak order $(\lsub{a}{G},\leq_{a})$.

(b)  For $a\in \Ob(G)$,  $\omega_{w_{a}}=\omega_{a}^{-1}$, $w_{w_{a}}=a$ and  $\omega_{a}(\tensor*[_{w_{a}}]{\Phi}{_{\mathrm{re}}^{+}})=\tensor*[_a]{\Phi}{_{\mathrm{re}}^{-}}$. For any morphism $g\colon a\to b$ in $G$, one has $\Phi_{g\omega_{a}}=\tensor*[_b]{\Phi}{_{\mathrm{re}}^{+}}\setminus \Phi_{g}$.

(c) Define a map $g\mapsto g^{\perp}\colon \lsub{a}{G}\to \lsub{a}{G}$ as follows: if  $g\colon b\to a$, let $g^{\perp}:=g\omega_{b}$. Then  $g\mapsto g^{\perp}$ is an order reversing bijection of $(\lsub{a}{G},\leq_{a})$ with itself, satisfying $(g^{\perp})^{\perp}=g$, $g\vee g^{\perp}=\omega_{a}$ and $g\wedge g^{\perp}=1_{a}$ (where these meets and joins exist even if the weak order at $a$ is not a lattice). For $g,h\in \lsub{a}{G}$, one  has $g\perp h\iff g\leq h^{\perp}$.

(d)  Assume that $R$ is complete. Then $R$ is rootoidal and for each $a\in \Ob(G)$, the weak order at $a$ is a complete ortholattice with maximum element $\omega_{a}:=\bigvee_{g\in \lsub{a}{G}}g$ and  orthocomplement  $g\mapsto g^{\perp}$ defined as in (c).  In particular, $R$ is antipodal.

(e) $R$ is complete if and only if it rootoidal and antipodal.

 (f) $R$ is complete (respectively, rootoidal, antipodal, preprincipal) if and only if    $R^{\mathrm{rec}}$ is complete (respectively, rootoidal,  antipodal, preprincipal).
\end{lemma}
\begin{proof}  (a) This follows from definition of weak order since $\tensor*[_{a}]{\Phi}{_{\mathrm{re}}^{+}}=\bigcup_{g\in \lsub{a}{G}}\Phi_{g}$.

(b) By (a), we have $\omega_{a}^{-1}(\tensor*[_{a}]{\Phi}{_{\mathrm{re}}^{+}})\subseteq \tensor*[_{w_{a}}]{\Phi}{_{\mathrm{re}}^{-}}$ and similarly with $a$ replaced by $w_{a}$.
This implies that  $(\omega_{w_{a}}^{-1}\omega_{a}^{-1})(\tensor*[_{a}]{\Phi}{_{\mathrm{re}}^{+}})\subseteq \tensor*[_{w_{w_{a}}}]{\Phi}{_{\mathrm{re}}^{+}}$. Also, $(\omega_{w_{a}}^{-1}\omega_{a}^{-1})(\tensor*[_{a}]{\Phi}{_{\mathrm{im}}^{+}})\subseteq \tensor*[_{w_{w_{a}}}]{\Phi}{_{\mathrm{im}}^{+}}$.
Hence $\Phi_{\omega_{a}\omega_{w_{a}}}=\emptyset$. Since $R$ is faithful, it follows  that $\omega_{a}\omega_{w_{a}}$ is an identity morphism and so $\omega_{w_{a}}=\omega_{a}^{-1}$.
In particular, $w_{w_{a}}=a$. Equality holds in the first inclusion in the proof of (b) since replacing $a$ by $w_{a}$ in it gives the reverse inclusion. To prove the equality $\Phi_{g\omega_{a}}=\tensor*[_b]{\Phi}{_{\mathrm{re}}^{+}}\setminus \Phi_{g}$, note both sides are contained in $\tensor*[_b]{\Phi}{_{\mathrm{re}}^{+}}$. If $\alpha$ is in this set of roots, one has $(g\omega_{a})^{-1}\alpha<0\iff
g^{-1}(\alpha)>0$, and the equality follows.

(c) By (b), we have $\Phi_{g^{\perp}}=
\Phi_{\omega_{a}}\setminus \Phi_{g}$. Hence the map $g\mapsto g^{\perp}$ is order reversing. It is an involution since $\Phi_{(g^{\perp})^\perp}=\Phi_{\omega_{a}}\setminus \Phi_{g^{\perp}}=\Phi_{g}$ and $R$ is faithful.  One has   $g\vee g^{\perp}=\omega_{a}$ and $g\wedge g^{\perp}=1_{a}$ since $\Phi_{g}\cup \Phi_{g^{\perp}}=\Phi_{\omega_{a}}$ and  $\Phi_{g}\cap \Phi_{g^{\perp}}=\emptyset$. The final assertion of (c) just amounts to $\Phi_{g}\cap \Phi_{h}=\emptyset\iff \Phi_{g}\subseteq \Phi_{\omega_{a}}\setminus \Phi_{h}$, which holds since $\Phi_{g}\subseteq \Phi_{\omega_{a}}$.

(d) Assume $R$ is complete. Then the weak order at $a\in \Ob(G)$ obviously  has maximum element $\omega_{a}:=\bigvee_{g\in \lsub{a}{G}}g$. The properties of the map $\perp$ in (c) show by definition (see \cite{bjornerbrenti})  that
the weak order is a complete ortholattice  with orthocomplement $\perp$.
It remains to prove $R$ is rootoidal.
Certainly each weak order, as a complete lattice, is a complete meet semilattice.  Suppose $a\in \Ob(G)$ and $g_{i},h$ in $\lsub{a}{G}$ satsify  $g_i\perp h$ for all $i$. Then  ${g_i}\leq_{a} h^{\perp}$ for all $i$. Hence the join $g=\bigvee_{i\in I}g_{i}$ satisfies $g\leq_{a}h^{\perp}$.
That is, $g\perp h$, as required to verify the JOP.

(e) The ``only if'' direction follows from (d). Conversely, suppose $R$ is rootoidal and antipodal. Then for any object $a$ of $G$,  the weak order at $a$ is a complete meet semilattice with a maximum element, which implies it is a complete lattice. Hence $R$ is complete by definition.

 (f) The groupoid and its weak orders are preserved under real compression. Hence for any property, such as those in (f),  which is expressible in terms of the groupoid and its weak orders, $R$ has that property if and only if $R^{\mathrm{rec}}$ does.
\end{proof}

 Some additional properties of complete faithful signed groupoid sets in \cite{rootoid1}--\cite{rootoid2}, notably the existence of an analogue of the diagram automorphism of  a finite Coxeter group corresponding to conjugation by the longest element,   can also be  extended to antipodal signed groupoid sets.

 \begin{example}\label{notionsexamp} (1) Consider the signed groupoid set  $R=(G,\lsub{\bullet}{\Phi},\lsub{\bullet}{\Phi}^+)$  from  the standard root system  of  a  Coxeter group
 (Example \ref{coxeterexample}).  It is easily verified to be principal, real and compressed (these conditions
 reduce to well known elementary  properties of  Coxeter groups and their root systems). The weak order at the unique object is the usual weak (right)
order of the Coxeter group, which is  known to be a complete  meet
semilattice (see \cite{bjornerbrenti}).   If $W$ is finite, it is well
known that the longest element $w_0$ is a maximum element of
weak order, which implies  that weak order is a complete lattice. Lemma  \ref{JOPlem}
 implies  that  $R$  is complete, antipodal and   rootoidal.  (Similarly, it follows from results
  in \cite{weylgroupoid5} that the signed groupoid sets discussed in  Example \ref{weylgroupoidexample}  are
  principal, real, compressed, complete, antipodal  and rootoidal.) The JOP for
  infinite  Coxeter groups $W$ is proved in   \cite{DyW}, showing that
  $R$ above is  principal, real, compressed and rootoidal (though not complete) in  general. This
  example and its relation to conjectures in   op. cit.  (which, from the discussion in \ref{convgeom}, are very closely related to various notions of convexity on root systems) provided the first author's principal motivation for the study of rootoids, which underlies this paper.

 (2)  Suppose that $R=(G,\lsub{\bullet}{\Phi},\lsub{\bullet}{\Phi}^+)$ is the  signed groupoid set in  Example \ref{coneexamp} associated to the set of all (real or imaginary roots) of a Kac-Moody Lie algebra with Weyl group $W$.  The real and imaginary roots of $R$ as defined here coincide with the real and imaginary roots in the usual sense of Kac-Moody Lie algebras, and parallelism classes of real roots are singletons. Then  $R^{\mathrm{rec}}$ identifies with the signed groupoid set associated similarly to the subsystem of all real roots, which is isomorphic (as sigend groupoid set) to that in (1) (with possibly infinite $W$).

  (3)  Suppose that $R=(G,\lsub{\bullet}{\Phi},\lsub{\bullet}{\Phi}^+)$ is the  signed groupoid set associated  in a similar way as in (1)  to a non-reduced crystallographic root system of a finite Weyl group $W$ (see \cite{Bour}). There are no imaginary roots and the  parallelism class of  a root is the set of all roots which are positive real  scalar multiplies of it. Then $R_\mathrm{rec}$ is isomorphic to the signed groupoid set attached as in (1) to a reduced root system of $W$.  \end{example}

    Part (f) of the following proposition  will play an important role in the proof of our main result. The crucial points, which characterize simplicial oriented geometries amongst oriented geometries by completeness properties, were  stated  in \cite[Theorem 6.11]{rootoid1}, and are essentially just a translation of results from \cite{hyperplane}. We include a  proof, citing \cite{hyperplane} for the key facts.

\begin{proposition}\label{acyctosgs}
 Let $R=(G,\Phi,\Phi^+)$ be the signed groupoid set $\SGS(A)$ associated to a  preacycloid $A=(E,*,\mathcal{T})$, where $\mathcal{T}\neq \emptyset$, in Example \ref{exampleacycloid}.

 (a) $R$ is finite, faithful,  connected, simply connected, antipodal  and inversion set finite (hence interval finite).

 (b) $R$ is real if and only if $A$ is loopless.

 (c) $R$ is real and compressed if and only if $A$ is simple.
 More generally, the signed groupoid set attached to the  simplification $A^{\circ}$ of $A$ canonically identifies with  the real compression $R^{\mathrm{rec}}$ of $R$.

  (d) $R$ is   preprincipal  if and only if $A$ is  an acycloid.

  (e)  $R$ is real and  principal if and only if $A$ is a simple acycloid.

 (f) Assume that $A$ is the  tope (pre)acycloid of  an oriented matroid $M$.   Then $R$ is finite, faithful, connected, simply connected,   preprincipal and   antipodal.  Further, $R$ is real and principal (equivalently, real and compressed) if and only if $M$ is an oriented geometry.  Finally, $R$ is complete  (or equivalently, rootoidal) if and only if
 $M$ is simplicial.
\end{proposition}
\begin{proof}
(a) Note that $R$ is  connected and simply connected by definition.

Suppose that $H,K\in \mathcal{T}$ and $u\colon \widetilde K\to \widetilde H$ is a morphism in $G$. Then \begin{equation*}
\Phi_u=(H\cup L^+)\cap -(K\cap L^+)=H\cap K^*=H\setminus K.
\end{equation*}
If $\Phi_u=\emptyset$, then $H\subseteq K$. Since $H\dotcup H^*=K\dotcup K^*=E\setminus L$, this forces $H=K$ and  $\widetilde H= \widetilde K$, proving $R$ is faithful.
This also shows  that  that the maximum element of weak order at $\widetilde H$ is the morphism $m:\widetilde{H^*}\to \widetilde{H}$, in $G$, since $\Phi_m=H\supseteq  \Phi_u$ for all $u$ as above.

Since $R$ is finite, it is inversion set finite and therefore interval finite  by subsection \ref{sgsterm}.

(b) From above, we see that the set of  positive imaginary roots of $R$ at any object $\widetilde H$ of $G$ is $L^+$, and that $L$ is  therefore the set of all imaginary roots.

(c) By (b), we may assume for the proof of the first assertion of (c) that $A$ is loopless (that is, $L=\emptyset$) and $R$ is real. From   the definitions and the above description of inversions sets, one sees  that for any $H\in \mathcal{T}$,   $\alpha,\beta\in \lsub{\widetilde{H}}\Phi$
are parallel in $\lsub{\widetilde H}{\Phi}$ if and only if they are parallel as elements of $E$ in the acycloid $A$.  We sketch the (essentially trivial)  proof of the second assertion.  To  form $R^{\mathrm{rec}}$, one discards its imaginary roots and identifies parallelism classes of real roots to singletons. To form $A^{\circ}$, one discards  loops of $A$ and identifies the parallelism classes of non-loops in $A$ to singletons. It is straightforward to check from the definitions and facts above that the construction of the associated signed groupoid set   is compatible with  these corresponding  deletions and identifications.

(d)  From the proof of (a), the morphism $u\colon \widetilde K\to \widetilde H$ in $G$ has inversion set $\Phi_u =H\cap K^*=H\setminus K$. By definition, then, the morphism $u\colon \widetilde K\to \widetilde H$ is atomic in $R$ if and only if $K\neq H$ and   for each $J\in \mathcal{T}$, $H\cap J^*\subseteq H\cap K^*$ implies
$H\cap J^*=\emptyset$  or $H\cap J^*=H\cap K^*$.
Note \begin{equation*}
J=(J\cap H)\cup (J\cap H^*)=(H\setminus (H\cap J^*))\cup
(H\cap J^*)^*
\end{equation*} and similarly for $K$.
Therefore, $u$ is atomic if and only if  $H\neq K$ and    the only $J\in \mathcal{T}$ which satisfy $H\cap J^*\subseteq H\cap K^*$ are $J=H$ and $J=K$.

 For any $J\in \mathcal{T}$,  $H\cap J^*$ is a union of parallelism classes for $A$ which are contained  in $E\setminus L$. So if  $e\in H$ and $K=(H\setminus [e])\cup [e]^*\in \mathcal{T}$, then $u$ is atomic, since $\Phi_u=[e]$ is a single such parallelism class.
The argument below will show that if $A$ is an acycloid or $R$ is preprincipal, every atomic  morphism $u$ in $R$ so arises.

Assume first  that $A$ is an acycloid. By (a),  $R$ is  interval finite. Suppose that $u\colon \widetilde K\to \widetilde H$ in $G$ is atomic (so $H\neq K$). Choose by (A4) some $e\in H\setminus K$ such that $J:=(H\setminus [e])\cup[e]^*$. Then $H\setminus J=[e]\subseteq H\setminus K$,
which implies that $K=(H\setminus [e])\cup [e]^*$ and $\Phi_u=[e]$.
 Now let  $u\colon \widetilde K\to \widetilde H$ be any atomic morphism, and  $v\colon \widetilde J\to \widetilde H$ be any morphism. The inversion set  $\Phi_v=J\cap H^*$  is a union of parallelism classes, while $\Phi_u$ is a single parallelism class,  so either $\Phi_u\supseteq  \Phi_v$ or $\Phi_u\cap \Phi_v=\emptyset$. This  shows $R$ is preprincipal.


 Conversely, suppose that $R$ is preprincipal.
 Suppose $u\colon \widetilde K\to \widetilde H$ is atomic.
 By definition of  preprincipalness, for any $v\colon \widetilde  J\to \widetilde H$ in $G$, either $\Phi_v\cap \Phi_u=\emptyset$ or $\Phi_u\subseteq \Phi_v$. That is, for any $J\in \mathcal{T}$,
one has either $(H\cap J^*)\cap (H\cap K^*)=\emptyset$  or $H\cap K^*\subseteq H\cap J^*$ (that is, either
$ J^*\cap (H\cap K^*)=\emptyset$  or $J^*\supseteq  H\cap K^*$).
This together with the definition of the parallelism of the roots  implies that $H\cap K^*$ is a single parallelism class:
$H\cap K^*=[e]$ for any   $e\in H\cap K^*\subseteq E\setminus L$, and so $K=(H\setminus [e])\cup [e]^*$.

Now we check axiom (A4) for acycloids.  Let $H$, $J$ be in $\mathcal{T}$  with $H\neq J$. Let $v\colon \widetilde J\to \widetilde H$  in $G$. Then there is some  atom $u\colon \widetilde K\to \widetilde H $ in $\lsub{\widetilde H}{G}$ so $\Phi_u\subseteq \Phi_v$.  That is, $H\setminus K\subseteq H\setminus J$. Choose any $e\in H\setminus K$.
We have $K=(H\setminus[e])\cup[e]^*\in \mathcal{T}$ where $e\in H\setminus J$, as required. This completes the proof of (d).

For the proof of (e), one may assume by (c) that $A$ is simple. The result may be proved using (d) and Lemma  \ref{sgsfacts}(e).

 We prove (f). Its second sentence follows from (a), (d), (e) and   Lemmas \ref{matroidisacycloid} and \ref{preandprincipal}.   The completeness   of $R$ is equivalent to $R$ being rootoidal by Lemma \ref{JOPlem}(e).  Finally, $R$ is complete if and only if, in the terminology of \cite{hyperplane}, the poset of regions (or topes), oriented from any fixed base region, is a (complete) lattice; this holds if and only if the oriented geometry is simplicial by   \cite[Theorem 6.3 and 6.5]{hyperplane}
 \end{proof}

\section{Generalized Brink-Howlett construction of signed groupoid sets}\label{gbhc}
In this section we will discuss the generalized Brink-Howlett construction for signed groupoid sets with certain properties. We will show that the construction preserves those favorable properties. This will play an important role in  the proof and applications of   our main theorem in the next section.

 In this section,
$R=(G,\Phi,\Phi^+)$ is a faithful signed groupoid set, frequently satisfying additional stated conditions.

\begin{definition}\label{squaredef}
A \emph{square} of $R$ is a quadruple $(x,w,y,z)$ of morphisms  of $G$   such that $xw=yz$ and $x(\Phi_w)=\Phi_y$:
$$\begin{CD}
a @>w>> b\\
@VVzV @VVxV\\
c @>y>> d
\end{CD}$$
\end{definition}

\begin{example}\label{squareexamp} Assume that $R$ is antipodal, and let $x\colon b\to d$ be a morphism. Then $(x,w,y,z)=(x,x^{-1}\omega_{d}, x^{\perp}, (x^{\perp})^{-1}\omega_{d})$ is a square.
For one has  $x(x^{-1}\omega_{d})=\omega_{d}= x^{\perp} ((x^{\perp})^{-1}\omega_{d})$,   and, by
Lemma \ref{JOPlem},  \begin{equation*}
x(\Phi_{x^{-1}\omega_{d}})=x(\tensor*[_{b}]{\Phi}{_{\mathrm{re}}^{+}}\setminus \Phi_{x^{-1}})=\tensor*[_{d}]{\Phi}{_{\mathrm{re}}^{+}}\setminus \Phi_{x}=\Phi_{x^{\perp}}
\end{equation*}
(where the second equality  follows by checking that $x(\tensor*[_{b}]{\Phi}{_{\mathrm{re}}^{+}}\setminus \Phi_{x^{-1}})\subseteq\tensor*[_{d}]{\Phi}{_{\mathrm{re}}^{+}}\setminus \Phi_{x}$ and $\tensor*[_{b}]{\Phi}{_{\mathrm{re}}^{+}}\setminus \Phi_{x^{-1}}\supseteq  x^{-1}(\tensor*[_{d}]{\Phi}{_{\mathrm{re}}^{+}}\setminus \Phi_{x})$).
\end{example}

\begin{remark} Let $H$ be a category (often a groupoid in applications).  Suppose given two functors $F_{1},F_{2}\colon H\to G$  and a natural transformation $\eta\colon F_{1}\to F_{2}$. We say that $\eta$ is a square natural transformation if for each morphism $f\colon p\to q$ in $H$, the commutative diagram
$$\begin{CD}
F_{1}p @>F_{1}f>> F_{1}q\\
@VV{\eta_{p}}V @VV\eta_{q}V\\
F_{2}p @>F_{2}f>> F_{2}q
\end{CD}$$ from the definition of a natural transformation gives a
 square $(\eta_{g}, F_{1}f,F_{2}f,\eta_{p})$. The properties of
 squares given below imply that there is a subcategory of the
 category of functors $H\to G$ with all objects but only square
 natural transformations as morphisms.  This observation underlies  basic constructions in the theory of functor rootoids,
 which has been a principal motivation for the development of the theory of rootoids surveyed  in \cite{rootoid1}--\cite{rootoid2}
 and for  the approach in  this paper.
 \end{remark}

\begin{lemma}\label{add}
$(x,w,y,z)$ is a square of $R$ if and only if $xw=yz$,  $\Phi_{x^{-1}}\cap \Phi_w=\emptyset$, $\Phi_x\cap \Phi_y=\emptyset$, $\Phi_{z}\cap \Phi_{y^{-1}}=\emptyset$ and $ \Phi_{z^{-1}}\cap \Phi_{w^{-1}}=\emptyset$.
\end{lemma}

\begin{proof}
Assume that $(x,w,y,z)$ is a square of a signed groupoid set $R=(G,\Phi,\Phi^+)$. Since $x(\Phi_w)=\Phi_y\subseteq \lsub{d}{\Phi}^+$, $\Phi_{x^{-1}}\cap \Phi_w=\emptyset$.  Similarly since $x^{-1}(\Phi_y)=\Phi_w\subseteq \lsub{b}{\Phi}^+$, $\Phi_x\cap \Phi_y=\emptyset$.

Suppose $\alpha\in \Phi_{z}\cap \Phi_{y^{-1}}$. We show that $wz^{-1}(\alpha)\in \lsub{b}{\Phi}^-$.  Otherwise $wz^{-1}(\alpha)\in \lsub{b}{\Phi}^+$,   $w^{-1}wz^{-1}(\alpha)=z^{-1}(\alpha)\in \lsub{a}{\Phi}^-$ and $xwz^{-1}(\alpha)=y(\alpha)\in \lsub{d}{\Phi}^-$. This contradicts the fact $\Phi_{w}\cap \Phi_{x^{-1}}=\emptyset$.  Now $-y(\alpha)\in \Phi_y$ and $-wz^{-1}(\alpha)\in \lsub{b}{\Phi}^+\setminus \Phi_w$.  But this contradicts $x^{-1}\Phi_y=\Phi_w$ as $x^{-1}(-y(\alpha))=-wz^{-1}(\alpha)$.  Hence $\Phi_{z}\cap \Phi_{y^{-1}}=\emptyset$.

Similarly one can prove that $\Phi_{z^{-1}}\cap \Phi_{w^{-1}}=\emptyset$.

Conversely assume that $xw=yz$, $\Phi_{x^{-1}}\cap \Phi_w=\emptyset$, $\Phi_x\cap \Phi_y=\emptyset$, $\Phi_{z}\cap \Phi_{y^{-1}}=\emptyset$ and  $\Phi_{z^{-1}}\cap \Phi_{w^{-1}}=\emptyset$.
We need to show that $x(\Phi_w)=\Phi_y$.  Since $xw=yz$, $\Phi_{x^{-1}}\cap \Phi_w=\emptyset$ and $ \Phi_{z}\cap \Phi_{y^{-1}}=\emptyset$, by Lemma \ref{sgsfacts1} (c) we have $\Phi_x\dotcup  x\Phi_w=\Phi_y\dotcup  y\Phi_z$. Since $\Phi_x\cap \Phi_y=\emptyset$,  $\Phi_y\subseteq x\Phi_w$. Suppose there exists $\beta\in x\Phi_w\cap y\Phi_z$.
Then $z^{-1}y^{-1}(\beta)=w^{-1}x^{-1}(\beta)<0$.  So $-z^{-1}y^{-1}(\beta)=-w^{-1}x^{-1}(\beta)>0$ and this root is contained in $\Phi_{w^{-1}}\cap \Phi_{z^{-1}}$,  a contradiction. Hence $x\Phi_w\cap y\Phi_z=\emptyset$. So we conclude that $\Phi_y=x\Phi_w$.
\end{proof}

\begin{lemma}\label{squaresym} Let  $x$, $y$, $z$ and $w$ be morphisms of $G$.

(a)   $(x,w,y,z)$ is a square if and only if $(y,z,x,w)$ is a square.

(b) $(x,w,y,z)$ is a square if and only if $(w, z^{-1}, x^{-1}, y)$ is a square.\end{lemma}
\begin{proof}  This follows immediately from Lemma \ref{add}\end{proof}
The above lemma  implies that the dihedral group of order $8$ acts naturally
on the squares of $G$ (and less precisely, on the characterizations of a fixed square; for example, Lemma \ref{squaresym} (a) implies that
$(x,w,y,z)$ is a square if and only if $yz=xw$ and $y(\Phi_z)=\Phi_x$.

\begin{lemma}\label{lemma:squareequiv}
$(x,w,y,z)$ is a square of $R$ if and only if $xw=yz$, $x\vee y=xw$ and $x^{-1}\vee w=x^{-1}y$.  (Here the joins are taken with respect to  the weak order of the morphisms at the corresponding object  and  exist in this special situation even without any assumption that the weak orders are lattices.)
\end{lemma}

\begin{proof}
Suppose that $(x,w,y,z)$ is a square. By Definition \ref{squaredef} and Lemma \ref{add}, we have  $\Phi_{x^{-1}}\cap \Phi_w=\emptyset$.  So one has $\Phi_{xw}=\Phi_x\dotcup  x\Phi_w=\Phi_x\dotcup  \Phi_y$.  Hence $x\vee y=xw$. The equality $x^{-1}\vee w=x^{-1}y$ can be proven similarly.
Conversely assume $(x,w,y,z)$ has the properties: $xw=yz$, $x\vee y=xw$ and $x^{-1}\vee w=x^{-1}y$.
Note that $\Phi_x\subseteq \Phi_{x\vee y}$. So $\Phi_x\subseteq \Phi_{xw}$.  By Lemma \ref{sgsfacts1} (c) $\Phi_{x^{-1}}\cap \Phi_w=\emptyset$.  Similarly from $x^{-1}\vee w=x^{-1}y$ we have $\Phi_x\cap \Phi_y=\emptyset$.  By $yz=y\vee x$ one obtains $\Phi_{y^{-1}}\cap \Phi_z=\emptyset$. By $wz^{-1}=w\vee x^{-1}$ one obtains $\Phi_{w^{-1}}\cap \Phi_{z^{-1}}=\emptyset$. Then again by Definition \ref{squaredef} and Lemma \ref{add}  this implies that $x\Phi_w=\Phi_y$.
\end{proof}

The following lemma is straightforward from the  definitions and Lemma \ref{add}.

\begin{lemma}\label{lemma:composition}
 Let $x,x',y,v,w,z,z'$ be morphisms of $R$. Suppose $xw=yz$,  $x'w'=y'x$. If any two of $(x,w,y,z), (x',w',y',x), (x',w'w,y'y,z)$ are squares, then the third is also a square.
$$\begin{CD}
c @>w>> d  @>w'>> a\\
@VVzV @VVxV @VVx'V\\
e @>y>> f @>y'>>  b
\end{CD}$$
\end{lemma}

\begin{proof}
Assume that $(x,w,y,z)$ and $ (x',w',y',x)$ are squares.
 Then $xw=yz$ and $y(\Phi_z)=\Phi_x$, and $x'w'=y'x$ and $y'(\Phi_x)=\Phi_{x'}$.  Hence
 $x'w'w=y'xw=y'yz$ and  $(y'y)(\Phi_z
)=y'(y\Phi_z)=y'(\Phi_x)=\Phi_x'$. Therefore $(x',w'w,y'y,z)$ is a square.
The other assertions can be proved similarly (or  reduced to this one using symmetry properties of squares).
\end{proof}

\begin{lemma}\label{lemma:squarejoin}
 Assume that $R$ is complete.  Suppose that $(x,w_i,y_i,z_i), i\in I$   is  a family of squares. Then there exists a square $(x,w,y,z)$ with $w=\bigvee_iw_i$ and $y=\bigvee_iy_i$.
$$\begin{CD}
a @>w>> b\\
@VVzV @VVxV\\
c @>y>> d
\end{CD}
\,\,\,\,\,\,\,
\begin{CD}
a_i @>w_i>> b\\
@VVz_iV @VVxV\\
c_i @>y_i>> d
\end{CD}$$
\end{lemma}

\begin{proof}
Note that since $R$ is complete, the join of a set of morphisms at a given object always exists.  Let $\bigvee_iw_i=w$ and $\bigvee_iy_i=y$.   We let $z=y^{-1}xw$.
 Next we  show that $xw=\bigvee_i xw_i$.
 Note $x^{-1}\perp w_i$ for all $i$ by Lemma \ref{add}. So by JOP (since $R$ is rootoidal) $x^{-1}\perp w$. So $x\leq xw=x\bigvee_i w_i$ by  Lemma \ref{sgsfacts1} (c).

So by Lemma \ref{sgsfacts} (a), $xw_i\leq xw$ and hence $\bigvee_i xw_i\leq xw$.  On the other hand, $xw_i\leq \bigvee_i xw_i$.  Also we have $x\leq xw_i\leq \bigvee_i xw_i$.  So by Lemma \ref{sgsfacts} (a) again $w_i\leq x^{-1}\bigvee_i xw_i$. So $w=\bigvee_i w_i\leq x^{-1}\bigvee_i xw_i$. By Lemma \ref{sgsfacts} (a) again $xw=x\bigvee_iw_i\leq \bigvee_i xw_i$. Therefore $xw=\bigvee_i xw_i$.

Now by Lemma \ref{lemma:squareequiv} we only need to show that $x\vee y=xw$ and $x^{-1}\vee w=x^{-1}y$.
Note that $xw=\bigvee_i xw_i=\bigvee_i (x\vee y_i)=x\vee \bigvee_i y_i=x\vee y$.  The second identity can be proved similarly.
\end{proof}

\begin{lemma}\label{lem:basicconstruction} Assume that $R$ is finite and   complete.
Suppose $x,y\in \lsub{d}{G}$ with $x\perp y$.  Then there exists a unique  morphism  $y'$ such that $y\leq y'$, $(x,w',y',z')$ is a square for some $w',z'$ and if $(x,w'',y'',z'')$ is a square with $y
\leq y''$, then $y'\leq y''$.
$$\begin{CD}
a @>w'>> b\\
@VVz'V @VVxV\\
c @>y'>> d
\end{CD}
\,\,\,\,\,\,\,
\begin{CD}
a' @>w''>> b\\
@VVz''V @VVxV\\
c' @>y''>> d
\end{CD}$$
\end{lemma}

\begin{proof}
Assuming the existence, the uniqueness of such a morphism is evident. Now we construct explicitly such a $y'$.
Let $y_0=y$. Then we define a sequence $\{y_i\},\{w_i\}$ as follows:
$$x\vee y_i=xw_i,\qquad x^{-1}\vee w_i=x^{-1}y_{i+1}.$$
We verify  that $x\perp y_i$, $x^{-1}\perp w_i$, $y_0\leq y_1\leq y_2\ldots$ and $w_0\leq w_1\leq w_2\leq \ldots$.
 Note that $x\leq x\vee y_i=xw_i$ so $x^{-1}\perp w_i$ by Lemma \ref{sgsfacts1} (c). Similarly $x^{-1}\leq x^{-1}\vee w_i=x^{-1}y_{i+1}$. So   $x\perp y_{i+1}$.
The inequalities follow from the fact $\Phi_{y_i}\subseteq x\Phi_{w_i}$ and $\Phi_{w_i}\subseteq x^{-1}\Phi_{y_{i+1}}$.
By finiteness   of $G$, for sufficiently large $i$, $y_i$ and $w_i$ stabilize. We take $y':=y_i$, $w':=w_i$  for large $i$  and $z':=y'^{-1}xw'$. Then by Lemma \ref{lemma:squareequiv}, $(x,w',y',z')$ is a square. Assume $(x,w'',y'',z'')$ is a square with $y_0=y\leq y''$.  Then  $xw_0=x\vee y_0\leq x\vee y''=xw''$. Since $x^{-1}\perp w_0,w''$, we have $w_0\leq w''$ by Lemma \ref{sgsfacts} (a).
Now we show  by induction that $w_n\leq w''$ and $y_n\leq y''$ for any $n$.   Taking $n$ sufficiently large will then show  that
$w'\leq w''$ and $y'\leq y''$.
 Assume that $w_i\leq w''$ and $y_i\leq y''$.  Note that \begin{equation*}
x^{-1}y_{i+1}=x^{-1}\vee w_i\leq x^{-1}\vee w''=x^{-1}y''.
\end{equation*} Since $x\perp y_{i+1},y''$,  we have  $y_{i+1}\leq y''$ by Lemma \ref{sgsfacts} (a).  Similarly,
 \begin{equation*}
 xw_{i+1}=x\vee y_{i+1}\leq x\vee y''=xw''.
 \end{equation*}  Since $x^{-1}\perp w_{i+1},w''$, we have $w_{i+1}\leq w''$ by Lemma \ref{sgsfacts} (a). Then we readily see $y'\leq y''$.
\end{proof}
We will  henceforward  denote $y'$ in the above lemma  as  $\square_yx$.

\begin{remark} \label{complete1}The above lemma  is a key   step in the proof of our main theorem and has important  consequences (adjointness properties) we don't go into here. The construction used in its proof  is a simple instance of what was referred to in \cite{rootoid2} as the ``zig-zag construction''; there is also a closely related ``loop construction''  mentioned there (corresponding to the situation above  when $b=d$ and one additionally requires $y'=w'$) which we do not use in this paper.
 Under completeness assumptions as imposed here, both can be regarded as  special cases (corresponding to a groupoid generated by a single arrow, either not  a loop or  a loop respectively) of a non-standard construction of certain adjoint functors in the context of fibered categories, as the first author  will discuss elsewhere. The proof of many basic facts about closure  of the category of rootoids under   constructions  discussed  in \cite{rootoid2} can,  under such completeness  assumptions,   be given  using this  general construction.
   \end{remark}

\begin{corollary}\label{cor:multijoin} Assume that $R$ is finite and  complete.
Suppose $x_1,x_2,\ldots,x_p, y\in \lsub{d}{G}$ with $x_i\perp y$ for each $1\leq i\leq p$. Then there exists a unique morphism $y'\in \lsub{d}{G}$ such that $y\leq y'$ and for each $1\leq i\leq p$ there exists a square $(x_i,w_i,y',z_i)$ for some $w_i,z_i$ and for any family of squares $(x_i,w_i',y'',z_i')$ with $y\leq y''$ one has $y'\leq y''$.
$$\begin{CD}
a_i @>w_i>> b_i\\
@VVz_iV @VVx_iV\\
c @>y'>> d
\end{CD}
\,\,\,\,\,\,\,\,\,
\begin{CD}
a_i' @>w_i'>> b_i\\
@VVz_i'V @VVx_iV\\
c' @>y''>> d
\end{CD}$$
\end{corollary}

\begin{proof}
First extend $\{x_i\}$ to an infinite sequence by requiring $x_i=x_j$ if $i\equiv j\pmod p$.
Let $y_0=y$ and $y_i=\square_{y_{i-1}}x_i$.  Then by Lemma \ref{lem:basicconstruction} we have $y_0\leq y_1\leq y_2\ldots$. Since $R$ is a finite signed groupoid set, this ascending chain must stabilize. Let $y':=y_i=y_{i+1}=\ldots$ for sufficiently large $i$. Then let $w_i=x_i^{-1}(x_i\vee y')$ and $z_i=y'^{-1}x_iw_i$. By our construction $(x_i,w_i,y',z_i)$ is a square.
Now we let $(x_i,w_i',y'',z_i')$ be another family of squares with $y\leq y''$.  Then by Lemma \ref{lem:basicconstruction} we see $y_i\leq y''$ for all $i$,
  so $y'\leq y''$.  Uniqueness of $y'$ is clear.
\end{proof}

We denote $y'$ in the above lemma by $\square_y(x_1,x_2,\ldots,x_p)$.

\begin{remark}
Note that it follows from the definition that for $y\perp x_i, 1\leq i\leq p$, $y=\square_y(x_1,x_2,\ldots,x_p)$ if and only if there exists a square $(x_i,w_i,y,z_i)$ for each $1\leq i\leq p$.
\end{remark}

\begin{corollary}\label{cor:composition} Assume that $R$ is
finite and  complete.
Suppose $x_1,x_2,\ldots,x_p$ and $ y$ are elements of  $\lsub{d}{G}$ with $x_i\perp y$  for all $i$.  Assume  that   $ y=\square_y(x_1,x_2,\ldots,x_p)$ and $u$ is a morphism such that $y\leq yu$ and $x_i\perp yu$  for all $i$.   Then $yu\leq \square_{yu}(x_1,x_2,\ldots,x_p)=y\square_u(z_1,z_2,\ldots,z_p)$.
$$\begin{CD}
@.a_i @>w_i>> b_i\\
@.@VVz_iV @VVx_iV\\
e@>u>>c @>y>> d
\end{CD}$$
\end{corollary}

\begin{proof}
We only need to show the equality $\square_{yu}(x_1,x_2,\ldots,x_p)=y\square_u(z_1,z_2,\ldots,z_p)$.
To see this, one notes that the inequalities $\square_{yu}(x_1,x_2,\ldots,x_p)\leq y\square_u(z_1,z_2,\ldots,z_p)$ and
$y^{-1}\square_{yu}(x_1,x_2,\ldots,x_p)\geq \square_u(z_1,z_2,\ldots,z_p)$
follow readily from Corollary \ref{cor:multijoin} and Lemma \ref{lemma:composition} (the composition of two squares is a square).
 Using Lemma \ref{sgsfacts} (a)  one sees   $\square_{yu}(x_1,x_2,\ldots,x_p)\geq y\square_u(z_1,z_2,\ldots,z_p)$, and therefore we have equality in this  as required.
\end{proof}

\subsection{}\label{bhconst}
 Next we introduce the generalized Brink-Howlett construction, which may be defined for any  faithful signed groupoid set, though its main  properties of interest here require stronger assumptions.  In \ref{bhconst}--\ref{bhowmain}, $R=(G,\Phi,\Phi^+)$ denotes   a faithful  signed groupoid set.   First, we define a groupoid $G^\Box$ as follows.

 The objects of $G^\Box$ are pairs $(a,X)$ where
$a\in \Ob  (G)$ and $X\subseteq \lsub{a}{G}$.  Note that $X$ must
be finite if $G$ is finite. If  $X=\{x\}$  is a singleton, we may write $(a,x)$ in  place of $(a,X)$.  Let $(b,Y)$ be another object of $G^\Box$. A
morphism $f\colon (a,X)\to(b,Y)$ in $G^{\Box}$
 is
 by definition a morphism $f: a\rightarrow b$ of $G$ such that there exists a (necessarily unique, by faithfulness of $R$)  bijection $\sigma_f\colon X\to Y$ for which $(f,g,\sigma_f(g),(\sigma_f(g))^{-1}fg)$ is a square for all $g\in X$. Diagramatically,
$$\begin{CD}
d_g @>g>> a\\
@VV(\sigma_f(g))^{-1}fgV @VVfV\\
d_{\sigma_f(g)}@>{\sigma_f(g)}>> b
\end{CD}$$
is a square for all $g\in X$, where $d_g$  (respectively $d_{\sigma_f(g)}$)  denotes the domain of a morphism $g$  (respectively $\sigma_f(g)$).   Equivalently,   we require  $f(\Phi_{g})=\Phi_{\sigma_f(g)}$for all $g\in X$.
Composition of morphisms in $G^{\Box}$ is  by composition of their underlying morphisms\footnote{ To accord with common conventions that each morphism determines its domain and codomain, one should strictly denote a morphism
$f\colon (a,X)\to(b,Y)$ as a tuple $(b,Y,f,a,X)$  and use composition  $(c,Z,g,b,Y)(b,Y,f,a,X)=(c,Z,gf,a,X)$, or use some similar artifice, but we shall not adopt such  cumbersome notation.} in $G$.  Lemmas \ref{squaresym} and \ref{lemma:squareequiv} imply that  that $G^\Box$ is a  groupoid.

We associate the following root system $\Psi$ to $G^{\Box}$.  For $(a,X)\in \Ob(G^\Box)$,   define  $\lsub{(a,X)}{\Psi}:=\lsub{a}{\Phi}$ and $\lsub{(a,X)}{\Psi}^+=\lsub{a}{\Phi}^+$.  The action of the
morphisms on the root systems  is inherited  naturally from that of $G$.   Then
$R^\Box:=(G^{\Box},\Psi,\Psi^+)$ is clearly  a  faithful  signed groupoid set, which is finite if $R$ is finite.  The signed groupoid set $R^\Box$ is  said to be obtained by applying the generalized Brink-Howlett construction to $R$. For $g\in \lsub{(a,X)}{G}^{\Box}$,  write the corresponding  inversion set as
 $\Psi_g:=\{\alpha\in \lsub{(a,X)}{\Psi}^+\mid g^{-1}(\alpha)<0\}$.

  Note that the cardinality  $\vert X\vert $ is constant as   $(a,X)$  ranges over the objects of any component of $G^\Box$.
The full subgroupoid of $G^\Box$ on all objects $(a,\emptyset)$ with $a\in \Ob(G)$ is therefore a union of components of $G^\Box$ and is canonically isomorphic to $G$. If $G$ is inversion set finite,  then the multiset $(\vert \Phi_g \vert \mid g\in X)$ is constant as   $(a,X)$  ranges over the objects of any component of $G^\Box$.
\begin{example}\label{brhow2examp}
 Suppose that $(W,S)$ is a  Coxeter group  with standard root system $\Phi$ (as in \cite{Hum} or \cite{bjornerbrenti}) and simple roots $\Pi$.
Let $(G,\lsub{\bullet}{\Phi},\lsub{\bullet}{\Phi}^+)$ be the (principal, rootoidal) signed groupoid set discussed in Example \ref{coxeterexample} and Example \ref{weylgroupoidexample} (1).  Recall that $G$ has one object ($\bullet$)
and $\Mor(G)=W$. The  objects  of  $G^\Box$  are pairs  $(\bullet, X)$ with  $X\subseteq W$. Let $H$ denote the full subgroupoid  of $G^\Box$ with objects of the form  $(\bullet,I)$ where $I\subseteq S$. Since $S=\{w\in W\mid \vert \Phi_w \vert =1\}$, the preceding paragraph implies that $H$ is a union of components of $G^\Box$.

For  $I\subseteq S$, let $\Pi_I$ be the set of simple roots such that the corresponding reflection is in $I$.  A morphism $(\bullet,I)\to (\bullet,J)$ in $H$, where $I,J\subseteq S$, is then just an element $w$ of $ W$ such that $w(\Pi_I)=\Pi_J$,
with composition of morphisms induced by multiplication in $W$.
Thus, $H$ is canonically isomorphic to the groupoid investigated by Brink and Howlett in \cite{bh} in their study of normalizers of parabolic subgroups of Coxeter groups (see Example \ref{brinkhowexamp}).
\end{example}

\begin{proposition}\label{bhprop}  Assume  that $R$ is finite and  complete. Then  the signed groupoid set  $R^\Box$
 is finite and  complete.
\end{proposition}

\begin{proof}
There is  a natural injective map of sets  $\lsub{(a,X)}{G}^{\Box}\rightarrow \lsub{a}{G}$.    This map induces an order-isomorphism of the weak order on $\lsub{(a,X)}{G}^{\Box}$
with its image, viewed as a subposet of weak order on $\lsub{a}{G}$, since   $\Psi_x=\Phi_x$ for $x\in\lsub{(a,X)}{G}^{\Box}$.   Let $x,y\in \lsub{(a,X)}{G}^{\Box}$.  Let $z=x\vee y$ in $\lsub{a}{G}$.  Then by Lemma \ref{lemma:squarejoin}, $z\in \lsub{(a,X)}{G}^{\Box}$ and hence is the join of $x,y$ in $\lsub{(a,X)}{G}^{\Box}$.   Therefore $(G^{\Box},\Psi,\Psi^+)$ is complete, and  the above injective,  order-preserving map  preserves  joins of all subsets of its domain. To show  $R^\Box$  is rootoidal we have to show that $\lsub{(a,X)}{G}^{\Box}$ satisfies JOP. But this  property is clearly inherited from $\lsub{a}{G}$.
\end{proof}

\begin{proposition}\label{bhowmain}
Assume that  $R$  is  finite, complete and preprincipal. Then  $R^\Box$  is   also finite, complete and   preprincipal.
\end{proposition}

\begin{proof}
Let $B$ be the set of atomic morphisms of  $R^\Box$  and $\lsub{(a,X)}{B}$ be the set of atomic morphisms at the object $(a,X)$. Let $A$ be the set of  atomic morphisms of  $R$  and $\lsub{a}{A}$ be the atomic morphisms of $R$ at the object $a$.
We first show that  if $X=\{g_1,g_2,\ldots, g_p\}$,  then  $$\lsub{(a,X)}{B}=\{\square_{s}(g_1,g_2,\ldots,g_p)\mid s\in \lsub{a}{A}, s\perp g_i,\,\text{for}\, 1\leq i\leq p\}.$$
Let   $r\in \lsub{(a,X)}{B}\subseteq \lsub{a}{G}$. Choose $s\in \lsub{a}{A}$ such that $s\leq r$.  By definition (of square) $r\perp g_i$ for all $i$.  Therefore $s\perp g_i$ for all $i$ as well. So $\square_s(g_1,g_2,\ldots,g_p)$ can be defined and $s\leq \square_s(g_1,g_2,\ldots,g_p)$.
$\square_s(g_1,g_2,\ldots,g_p)$ is a morphism in  $\lsub{(a,X)}{G}^{\Box}$.  By Corollary \ref{cor:multijoin}, $\square_s(g_1,g_2,\ldots, g_p)\leq r$. But $r$ is an atomic morphism. So  this inequality is actually an equality. Therefore we proved that the left hand side  (that is, $\lsub{(a,X)}{B}$)   is contained in the right hand side.

Conversely, take $s\in \lsub{a}{A}$ such that $s\perp g_i$ for all $i$.  Let $r=\square_s(g_1,g_2,\ldots,g_p)$.  Before we prove the reverse inclusion, we show that for any
$y\in \lsub{(a,X)}{G}^{\Box}$,  we have either  $\Psi_r\subseteq \Psi_y$ or $\Psi_r\cap \Psi_y=\emptyset$.  If $s\leq y$, then by Corollary \ref{cor:multijoin} again,
$r\leq y$ (i.e. $\Psi_r\subseteq \Psi_y$.) Otherwise $\Phi_s\cap \Phi_y=\emptyset$.  So  $y^{-1}s\geq y^{-1}$  thanks to Lemma \ref{sgsfacts1} (c).  Then Corollary \ref{cor:composition} together with Lemma \ref{squaresym}  says that  $y^{-1}s\leq y^{-1}r$. So $y^{-1}\leq y^{-1}r.$ Hence $\Psi_y\cap \Psi_r=\emptyset$ by Lemma \ref{sgsfacts1} (c) again.

Now we show the reverse inclusion (i.e. $r\in \lsub{(a,X)}{B}$). We may take $u\in \lsub{(a,X)}{B}$ such that $u\leq r$  (note $r\neq 1_{(a,X)}$ since $s\leq r$ in $\lsub{a}{G}$). But the above paragraph  with $y=u$ implies   that $u=r$. So we are done.
Along the way we have also showed the atomic morphisms of
 $R^\Box$  have the properties of a preprincipal signed groupoid set.
\end{proof}

\begin{lemma}\label{gbhim}  Let  $a\in \Ob(G)$ and   $X\subseteq \lsub{a}{G}$. Then

(a)   $\tensor*[_{a}]{\Phi}{_{\mathrm{im}}^{+}}\subseteq \tensor*[_{(a,X)}]{\Psi}{_{\mathrm{im}}^{+}}$.

(b)   $\bigcup_{x\in X}\Phi_{x}\subseteq \tensor*[_{(a,X)}]{\Psi}{_{\mathrm{im}}^{+}}$.

(c) If $R$ is rootoidal and $X_{0}\subseteq X$ is such that the join
$x_{0}:=\bigvee X_{0}$ exists in $  \lsub{a}{G}$ (for instance, $R$ is complete and $X_{0}=X$), then $\Phi_{x_{0}}\subseteq \tensor*[_{(a,X)}]{\Psi}{_{\mathrm{im}}^{+}}$.

(d) If $X=\{x\}$ is a singleton and $R$ is antipodal, then
$\tensor*[_{(a,X)}]{\Psi}{_{\mathrm{im}}^{+}}=\Phi_{x}\dotcup\tensor*[_{a}]{\Phi}{_{\mathrm{im}}^{+}}$ and
the component $R^{\Box}[(a,X)]$ is an  antipodal signed groupoid set.
\end{lemma}
\begin{proof}
(a) We have \begin{equation*}
\tensor*[_{(a,X)}]{\Psi}{_{\mathrm{re}}^{+}}=\bigcup_{g\in \lsub{(a,X)}{G}^{\Box}}\Psi_{g}\subseteq\bigcup_{g\in \lsub{a}{G}}\Phi_{g}=\tensor*[_{a}]{\Phi}{_{\mathrm{re}}^{+}}
\end{equation*} and (a) follows by taking complements in $\lsub{a}{\Phi}^{+}$.

(b) If $g\colon (b,Y)\to (a,X)$ is in  $\lsub{(a,X)}{G}^{\Box}$, then
there is a bijection $\sigma\colon X \to Y$ such that for each $x\in X$,  $(g,\sigma(x),x,g_{x})$ is a square for some morphism $g_{x}$ of $G$. In particular, if $x\in X$, we have $g\perp x $ in $\lsub{a}{G}$ by Lemma \ref{add}.
That is, $\Psi_{g}\cap \Phi_{x}=\Phi_{g}\cap \Phi_{x}=\emptyset$ and so
$\Phi_{x}\cap \tensor*[_{(a,X)}]{\Psi}{_{\mathrm{re}}^{+}}=\emptyset$. Therefore  $\Phi_{x}\subseteq \tensor*[_{(a,X)}]{\Psi}{_{\mathrm{im}}^{+}}$ (see the proof of (a)), proving (b).

(c) This follows on noting in the proof of (b) that $g\perp x $
for all $x\in X_{0}$ implies $g\perp x_{0}$, by the JOP, and arguing similarly as in the end of the proof of (b).

(d)  By definition of morphisms in $R^{\Box}$, for any object $(b,Y)$ in the same component as $(a,X)$,  the set $Y$ is a singleton. Hence it suffices to  show that
$\tensor*[_{(a,X)}]{\Psi}{_{\mathrm{im}}^{+}}=\Phi_{x}\cup\tensor*[_{a}]{\Phi}{_{\mathrm{im}}^{+}}$ and that
$\lsub{(a,x)}G^{\Box}$ has a maximum element in weak order.
  Write  $g:=x^{\perp}\colon b\to a$ in $\Mor(G)$ in Lemma \ref{JOPlem}. By  Example  \ref{squareexamp},    $g$ provides a  morphism $g\colon (b,y)\to(a,x)$    in $\lsub{(a,x)}G^{\Box}$ with $y=(x^{\perp})^{-1}\omega_{a}$. One has $\Psi_{g}=
\tensor*[_{a}]{\Phi}{_{\mathrm{re}}^{+}}\setminus  \Phi_{x}$. Hence  $g$ is necessarily the maximum element in weak order in $\lsub{(a,x)}G^{\Box}$ since the complement of $\Psi_{g}$  in $\tensor*[_{(a,X)}]{\Psi}{^{+}}$ is contained in  $\tensor*[_{(a,X)}]{\Psi}{_{\mathrm{im}}^{+}}$ (see Lemma \ref{JOPlem} (a)).
\end{proof}

 \subsection{} \label{comp}
 In \ref{comp}--\ref{bhmain2},   assume that $R=(G,\Phi,\Phi^+)$ is a faithful,  connected, simply connected signed groupoid set.    Since for any objects $a$, $b$ in $G$, there is a unique morphism $a\to b$ in $G$, and since it is invertible, we may use the map  $\lsub{a}{\Phi} \to \lsub{b}{\Phi}$ given by action of this morphism  to canonically identify
 $\lsub{a}{\Phi}$ and  $\lsub{b}{\Phi}$ for all  $a,b\in \Ob(G)$. We thereby identify all $\lsub{a}{\Phi}$ with a single set $\Phi$, and  identify the function  $\lsub{a}{\Phi}\to \lsub{b}{\Phi}$ induced by action of the  groupoid morphism
$a\to b$  with  the  identity map on $\Phi$.
 Note however that  $\lsub{a}{\Phi}^+\subseteq \Phi$  still depends on $a$ under this identification.

  Using the identification  $\Phi=\lsub{a}{\Phi}$ for any $a\in \Ob(G)$, we may unambiguously transfer the relations  of dominance order $\preceq_a$ and parallelism $\sim_a$, and the subsets of real and imaginary roots, from $\lsub{a}{\Phi}$ to $\Phi$.
 We denote them as $\preceq$, $\sim$, $\Phi_{\mathrm{re}}$ and $\Phi_{\mathrm{im}}$ respectively.

  The real compression $R_\mathrm{rec}=(G,\Phi_{\mathrm{rec}},\Phi_{\mathrm{rec}}^+)$ is  connected and simply connected, so
 one may similarly identify all its systems $\lrsub{a}{\Phi}{\mathrm{rec}}$ with a single definitely  involuted  set ${\Phi}_{\mathrm{rec}}$
 on which all groupoid morphisms act by the identity map.  Concretely,
$\Phi_{\mathrm{rec}}:=\Phi_\mathrm{re}/\negthinspace\sim=\{[\alpha]\mid \alpha\in \Phi_{\mathrm{re}}\}$ where $[\alpha]$ denotes  the  parallelism class of $\alpha$ in  $\Phi$. Also, $\Phi_{\mathrm{rec}}$ is a definitely   involuted set with $-[\alpha]=[-\alpha]$ for $\alpha\in
\Phi_{\mathrm{re}}$. Finally,  we have
$\tensor*[_{a}]{\Phi}{_{\mathrm{rec}}^+}:= \{[\alpha]\mid \alpha\in
\lrsub{a}{\Phi}{\mathrm{re}}\cap \lsub{a}{\Phi}^+\}$ for all $a\in \Ob(G)$.

The following result follows immediately from Proposition \ref{bhprop}, Proposition \ref{bhowmain} and Lemma \ref{preandprincipal}.

\begin{corollary}\label{bhmain2}
 Let $R=(G,\Phi,\Phi^+)$ be a  finite, connected, simply connected,  preprincipal and  complete signed groupoid set, and $S$ denote  any connected component of $R^\Box$. Then $S$ is also a finite, connected, simply connected, preprincipal and  complete signed groupoid set. Further, the real compression has the properties of  $S$  listed above, but  is also real and principal.
\end{corollary}

\begin{example}\label{gbhscexamp} We now discuss the generalized Brink Howlett construction when  $R$ is (faithful),  connected and  simply connected.
We consider a component $R^{\Box}[(a,X)]$ where $a\in \Ob(G)$ and $X\subseteq \lsub{a}{G}$. Write $R^{\Box}[(a,X)]=(H,\Psi,\Psi^{+})$.
 For any morphism
$g\colon c\to a$ in $\lsub{a}{G}$ (for instance, $g\in X$),
we have $\Phi_{g}=\lsub{a}{\Phi}^{+}\cap-\lsub{c}{\Phi}^{+}$
or equivalently,
$\lsub{c}{\Phi}^{+}=(\lsub{a}{\Phi}^{+}\setminus \Phi_{g})\dotcup -\Phi_{g}$. In these formulae,
 terms of the form  $\lsub{u}{\Phi}^{+}$ can be equivalently replaced throughout by $\tensor*[_{u}]{\Phi}{_{\mathrm{re}}^{+}}$.

 Given $g:c\rightarrow a$ as above,   for any morphism $h\colon d\to b$ in $G$, there is a unique commutative diagram
$$\begin{CD}
c @>g>> a\\
@VV kV @VVfV\\
d@>{h}>> b
\end{CD}$$
since $G$ is connected and simply connected. Since $f$ acts trivially on $\Phi$, this diagram  is a square of $R$ if and only if
$\Phi_{h}=\Phi_{g}$,  or equivalently, if and only if $\Phi_{g}\subseteq \lsub{b}{\Phi}^{+}$ and
$\lsub{d}{\Phi}^{+}=(\lsub{b}{\Phi}^{+}\setminus \Phi_{g})\dotcup -\Phi_{g}$.  In particular, for fixed $g$ and $b$,  a square of $R$
as above is unique if it exists: $f$ would have to be   the
unique morphism $a\to b$ and $h$ would have to be the morphism, if such exists, in $\lsub{b}{G}$ with $\Phi_{h}=\Phi_{g}
$, which would be unique by faithfulness of $R$.   Still for fixed
$g$ and $b$, such a square exists if and only if   $\Phi_{g}\subseteq
\lsub{b}{\Phi}^{+}$ and
$(\lsub{b}{\Phi}^{+}\setminus \Phi_{g})\dotcup -\Phi_{g}=\lsub{d}{\Phi}^{+}$ for some (necessarily unique) $d\in \Ob(G)$, in which case  $h$ is the unique morphism $d\to b$ in $\Ob(G)$.

It follows that for $b\in \Ob(G)$,  there is a morphism $f\colon (a,X)\to (b,Y)$ in $H$, for some  $Y\subseteq \lsub{b}{G}$, if and only if
for each $g\in X$, there is some (necessarily unique)  object $u_{g}$ of $G$ such that $\Phi_{g}\subseteq
\lsub{b}{\Phi}^{+}$ and
$(\lsub{b}{\Phi}^{+}\setminus \Phi_{g})\dotcup -\Phi_{g}=\lsub{u_{g}}{\Phi}^{+}$. Then $Y$ is uniquely determined as
\begin{equation*}
Y=\{h\in \lsub{b}{G}\mid h\colon u_{g}\to b \text{ \rm for some
$g\in X$}\}
\end{equation*} or  by
\begin{equation*}
Y=\{h\in \lsub{b}{G}\mid \Phi_{h}=\Phi_{g}\text{ \rm for some $g\in X$}\}
\end{equation*}
and $f\colon (a,X)\to (b,Y)$ is the unique moprhism $f\colon a \to b$  in $G$.

In particular,  there is a faithful  functor $H$ to $G$ which maps  $(b,Y)\mapsto b$ on objects of $H$ and sends a morphism $f\colon(b,Y)\to (c,Z)$ in $H$ to $f\colon b\to c$ in $G$.
One may use this to canonically embed $H$  as a subgroupoid of $G$.

  \end{example}
  \subsection{}  We shall use the following notation and terminology. For any faithful signed groupoid set $R=(G,\Phi,\Phi^{+})$ and $X\subseteq \lsub{a}{G}$, write
 $R\dslash (a,X):=R^{\Box}[(a,X)]$ and call it the \emph{
 hypercontraction} of $R$ at $(a,X)$.
 (If $X\neq \emptyset$, we may  use the convention that morphisms determine their codomains and domains to write this more compactly as $R\dslash X$. If $X=\{g\}$, we may also write  $R\dslash g$ for $R\dslash X$.)  We commend to the reader, as an instructive exercise,  checking  the fact (which is not used in this paper) that a hypercontraction of a hypercontraction of $R$ is (canonically isomorphic to) a hypercontraction of $R$.

  We call a signed groupoid set of the form $R\dslash (a_{0},\emptyset)\dslash s_{1} \ldots \dslash s_{n}$, where $a_{0}$ is an object of $R$ and
 $s_{i}$ is an atomic morphism of    $R\dslash (a_{0},\emptyset)\dslash s_{1}\dslash \ldots \dslash s_{i-1}$  for $i=1,\ldots,n$, a \emph{quasicontraction} of $R$. If $n=0$, this quasicontraction is just the component  $R\dslash(a_{0},\emptyset)=R[a_{0}]$  of $R$ and  if  $n>0$, it is   equal to
 $R\dslash s_{1}\dslash \ldots \dslash s_{i-1}$; in particular, quasicontractions are always connected.  We call  $R\dslash s$, where $s$ is an atomic morphism of $R$, an \emph{elementary quasicontraction} of $R$. Thus,
 the quasicontractions of $R$  form  the smallest set  $\mathfrak{Q}$  of signed groupoid sets with the properties that any component of $R$ is in $\mathfrak{Q}$ and $\mathfrak{Q}$ is closed under the formation of elementary quasicontractions.

  \begin{remark} The above development shows that hypercontraction (and hence quasicontraction) preserves the following subclasses  of faithful, connected  signed groupoid sets:
 (1)  the finite complete ones (2)  the finite, complete,  and preprincipal ones, and  (3) the simply connected ones.
It  also preserves the classes of (4) finite ones  (5) complete ones (6) rootoidal ones and  (7) rootoidal and preprincipal ones. Preservation of class (4) is trivial, and that of classes (5)--(7) can be proved by similar  but slightly  more technical arguments to those  in this paper, or deduced as an application of the  theory of functor rootoids sketched
in \cite{rootoid1}--\cite{rootoid2} (for which proofs will be given elsewhere, though the main  ideas  for the original  proofs all appear either  in op. cit., in this paper or in \cite{DyW}).
  \end{remark}

\section{Main Result and the Proof}

In this section we  state and prove  our main theorem that a finite\footnote{The reader may show as an exercise that ``finite'' is redundant in this statement, but we retain it for emphasis.}, connected, simply connected, real,  principal, complete signed groupoid set has the structure of a simplicial oriented geometry.   Throughout this section, we adopt the conventions of \ref{comp} regarding any connected, simply connected signed groupoid set: the sets of  roots  at the various objects are all canonically identified so all groupoid elements  act by identity maps.

\begin{lemma}\label{acycloidlemma}
 Let $R=(G,\Phi,\Phi^+)$ be a faithful, finite, connected, simply connected   signed groupoid set.  Let
$A:=(\Phi,*,\mathcal{T})$  where  $\mathcal{T}=\{\tensor*[_a]{\Phi}{_{\mathrm{re}}^{+}}\mid a\in \Ob  (G)\})$ and  $*\colon \Phi\to \Phi$ is the map $x\mapsto -x$.

(a)  $A$ is a  preacycloid if and only if $R$ is antipodal. This holds in particular if  $R$ is complete.

(b)  Assume that $A$ is a preacycloid. Then  $R$ is   isomorphic to the signed groupoid set $\SGS (A)$ attached to $A$ in Example \ref{exampleacycloid} (taking $L^{+}:=\tensor*[_{a}]{\Phi}{_{\mathrm{im}}^+}$). Hence properties of $A$ and $R$ are related as in Proposition \ref{acyctosgs}.
\end{lemma}

\begin{proof}
 (a)  We consider the conditions (A1)--(A3) for $A$ to be a preacycloid.
Condition  (A1) is trivial. Condition  (A2) holds with $L:=\Phi_{\mathrm{im}}$. If the  weak order of $R$ at $a$ has a maximum element $\omega_{a}\colon w_{a}\to a$, then  Lemma \ref{JOPlem}(b)  implies that for each $a\in \Ob(G)$,
one has $(\tensor*[_a]{\Phi}{_{\mathrm{re}}^{+}})^{*}=\tensor*[_{w_{a}}]{\Phi}{_{\mathrm{re}}^{+}}$, so (A3) holds.
On the other hand, if (A3) holds, then for any $a\in \Ob(G)$, there is some $w_{a}\in \Ob(G)$ such that $\tensor*[_{w_{a}}]{\Phi}{_{\mathrm{re}}^{+}}=
(\tensor*[_a]{\Phi}{_{\mathrm{re}}^{+}})^{*}$. Letting
$\omega_{a}$ be the unique morphism $w_{a}\to a$ in $G$, we then have $\Phi_{\omega_{a}}=\tensor*[_a]{\Phi}{_{\mathrm{re}}^{+}}$
and so $\omega_{a}$ is the  maximum element of  the weak order at $a$ by Lemma \ref{JOPlem}(a). The final assertion of (a) follows from Lemma \ref{JOPlem}(d).

(b) Note that $G$ has at least one object since it is connected, and so $\mathcal{T}\neq \emptyset$.  Thus the signed groupoid set  $R'$ attached to $A$  in Example \ref{exampleacycloid} is defined.  Its underlying groupoid is a  connected, simply connected groupoid
with objects $\widetilde H$ for $H\in \mathcal{T}$, acting trivally on the strictly involuted set $\Phi$, with $\lsub{\widetilde H}{\Phi}^{+}=H\cup \Phi_{\mathrm{im}}^{+}$. Note that  the map $\Ob(G)\to \mathcal{T}$ given by $a\mapsto \tensor*[_a]{\Phi}{_{\mathrm{re}}^{+}}$ is a bijection. For if $\tensor*[_a]{\Phi}{_{\mathrm{re}}^{+}}=\tensor*[_b]{\Phi}{_{\mathrm{re}}^{+}}$, where $a,b\in \Ob(G)$,  then the morphism $u\colon a\to b$ in $G$ satisfies $\Phi_{u}=\emptyset$ and so $u$ is an identity morphism, by faithfulness of $R$, and $a=b$. Since $\tensor*[_a]{\Phi}{_{\mathrm{re}}^{+}}\cup \Phi_{\mathrm{im}}^{+}=\lsub{a}{\Phi}^{+}$, it easily follows  that  this bijection induces an isomorphism of $R$ with $R'$.
 \end{proof}

 \subsection{} Let $\mathcal{R}$ be the class of all faithful, finite, connected, simply connected, antipodal signed groupoid sets.
Let $\mathcal{A}$ denote the class of all acycloids with at least one (possibly empty) tope. We note that the elementary  quasicontraction  of an acyloid in  $\mathcal{A}$ may have no topes, so  $\mathcal{A}$ is not closed under quasicontractions.

 For any  $R=(G,\Phi,\Phi^+)$ in $\mathcal{R}$, denote the preacycloid  $(\Phi,*,\{\tensor*[_a]{\Phi}{_{\mathrm{re}}^{+}}\mid a\in \Ob  (G)\})$ in $\mathcal{A}$ constructed from $R$ in Lemma   \ref{acycloidlemma} as   $\PA(R)$.
 For any $A$ in $\mathcal{A}$, let $\SGS(A)$ be the finite, faithful, connected, simply connected, signed groupoid set  attached to $A$ in Example \ref{exampleacycloid}.

 \begin{proposition}\label{invprop}  The maps $\SGS\colon \mathcal{A}\to \mathcal{R}$ and $\PA\colon \mathcal{R}\to \mathcal{A}$ induce inverse bijections  between the set of  isomorphism classes of preacycloids in $\mathcal{A}$ and   the set of isomorphism classes of  signed groupoid sets in $\mathcal{R}$.
\end{proposition}
\begin{proof} By Lemma \ref{acycloidlemma} (b), it suffices to show that  if  $A$ is a preacyloid in $\mathcal{A}$ and
 $R=\SGS(A)$, then $\PA(R)=A$.  Write $A=(E,*,\mathcal{T})$. Let  $L=L^{+}\dotcup(L^{+})^{*}$ be  the set of loops of $A$.  Write
 $R=(G,\Phi,\Phi^{+})$. By definition, $\Phi=E$ as involuted set,
 so $\PA(R)=(E,*,\mathcal{T}')$ where $\mathcal{T}'=\{\tensor*[_a]{\Phi}{_{\mathrm{re}}^{+}}\mid a\in \Ob  (G)\})$.
 By definition,  $\Ob(G)=\{\widetilde{H}\mid H\in \mathcal{T}\}$  and
 $\tensor*[_{\widetilde{H}}]{\Phi}{^{+}}=H\dotcup L^{+}$.
 Since $H\cup H^{*}=E\setminus L$ and $H^{*}\in \mathfrak{T}$ for all $H\in \mathfrak{T}\neq \emptyset$, it follows that $L^{+}=   \Phi^{+}_{\mathrm{im}}$. Therefore
 $\tensor*[_{\widetilde H}]{\Phi}{_{\mathrm{re}}^{+}}= \tensor*[_{\widetilde H}]{\Phi}{^{+}}\setminus \tensor*[_a]{\Phi}{_{\mathrm{im}}^{+}}=H$. Hence
 $\mathfrak{T}'=\mathfrak{T}$  as required.
 \end{proof}

\begin{lemma}\label{corresp}
Let $R=(G,\Phi,\Phi^+)$ be  in $\mathcal{R}$.

(a) Let   $g\in \Mor(G)$. Then  $R\dslash g$ is in $\mathcal{R}$ and
  $\PA (R\dslash g)= \PA (R)\dslash\Phi_{g}$  as preacycloids, where $\PA (R)\dslash\Phi_{g}$ is as  defined in subsection \ref{cont}. In particular, $\mathcal{R}$ is closed under quasicontractions.

(b) Assume  that $R$ is preprincipal (or equivalently, $\PA (R)$ is an acycloid). Then  the elementary quasicontractions of $\PA (R)$ are precisely the preacycloids  $\PA(R\dslash s)$  where $s$ is  an atomic morphism of $R$.
 \end{lemma}

\begin{proof} (a)  By Lemma \ref{gbhim}(d),  $R\dslash g$ is in $\mathcal{R}$, and denoting it   as $(H,\Psi, \Psi^{+})$, we have $\Psi=\Phi$ as involuted set and $\Psi_{\mathrm{im}}^{+}=\Phi_{\mathrm{im}}^{+}\dotcup \Phi_{g}$.
We have $\PA (R)=(\Phi,*,\mathfrak{T})$ where $\mathfrak{T}=\{\tensor*[_{a}]{\Phi}{^{+}_{\mathrm{re}}}\mid a\in \Ob(G)\}$.
Similarly,  $\PA (R\dslash g)=(\Phi,-, \mathfrak{T}')$
where $\mathfrak{T}'=
\{\tensor*[_{(a,h)}]{\Psi}{^{+}_{\mathrm{re}}}\mid (a,h)\in \Ob(H)\}$. Let $\Gamma:=\Phi_{g}$. By Example \ref{gbhscexamp}, $(a,h)\in \Ob(H)$ if and only if $a\in \Ob(G)$
is such that $\Gamma\subseteq
\tensor*[_{a}]{\Phi}{^{+}_{\mathrm{re}}}$ and
$(\tensor*[_{a}]{\Phi}{^{+}_{\mathrm{re}}}\setminus \Gamma)\cup -\Gamma=\tensor*[_{d}]{\Phi}{^{+}_{\mathrm{re}}}$ for some $d\in \Ob(G)$; in that case,  $\Phi_g=\Phi_h$ and \begin{equation*}
\tensor*[_{(a,h)}]{\Psi}{^{+}_{\mathrm{re}}}=
\tensor*[_{(a,h)}]{\Psi}{^{+}}\setminus \tensor*[_{(a,h)}]{\Psi}{_{\mathrm{im}}^{+}}= \tensor*[_{a}]{\Phi}{^{+}}\setminus (\tensor*[_{a}]{\Phi}{_{\mathrm{im}}^{+}}\dotcup \Phi_{g})=\tensor*[_{a}]{\Phi}{^{+}_{\mathrm{re}}}\setminus\Gamma.
\end{equation*} So, in terms of $\mathfrak{T}$, we have
\begin{equation*}
\mathfrak{T}'=\{H\setminus \Gamma\mid \text{ \rm $ H\in \mathfrak{T}$, $\Gamma\subseteq H$ and $(H\setminus \Gamma)\cup \Gamma^{*}\in \mathfrak{T}$}\}=\mathfrak{T}_{\Gamma}.
\end{equation*} where $\mathfrak{T}_{\Gamma}$ is as in the definition of $\PA (R)\dslash \Gamma$ in subsection \ref{cont}. Taking $g$ to be an atomic morphism of $R$  proves  that   $\mathcal{R}$ is closed under elementary quasicontractions, and hence it is closed under quasicontractions, proving (a).

 (b) The equivalence of the two asumptions is from Proposition \ref{acyctosgs}(d). The conclusion follows from (a) and the definition of elementary quasicontraction, since the parallelism classes in $\Phi_{\mathrm{re}}$ are the inversion sets of  atomic  morphisms, by Lemma \ref{preandprincipal}(a) and the fact groupoid morphisms act trivially on $\Phi$.
\end{proof}

We say that  $R$ in $\mathcal{R}$ and   $A$   in $\mathcal{A}$ (or their respective isomorphism classes) \emph{correspond} if   $R\cong \SGS(A)$ (or equivalently,   $A\cong\PA(R)$).   Part (c) of the following proposition may be viewed as a reformulation of
Handa's characterization of oriented matroids.

\begin{proposition} \label{handareform}
Suppose that  $R$  in $\mathcal{R}$ and   $A$   in $\mathcal{A}$ correspond.

(a) $A$ is simple if and only if $R$ is real and compressed.
More generally, the simplification of $A$ corresponds to the real compression of $R$.

(b) $A$ is an acycloid if and only if $R$ is preprincipal.
In that case, the  isomorphism classes of elementary quasicontractions of $R$ correspond bijectively to the  isomorphism classes of  elementary quasicontractions of $A$.

(c) $A$ is the tope (pre)acycloid of an oriented matroid if and only if every quasicontraction of $R$ is preprincipal.
 \end{proposition}

 \begin{proof}  (a) This follows from   Lemma  \ref{acyctosgs} (c).

(b) This follows from Proposition  \ref{acyctosgs}(d) and Lemma \ref{corresp}(b). Note however that the corresponding statement with
``elementary'' omitted does not follow, since  $R$  may have a contraction which is not preprincipal, and to which Lemma  \ref{corresp}(b) need not apply.

(c) By Theorem \ref{handa}, $A$ is a preacycloid attached to an oriented matroid  if and only if every quasicontraction of $A$ is an acycloid. By (b), this holds if and only if every quasicontraction of $R$ is preprincipal.
\end{proof}

Let us define an arbitrary signed groupoid set $R$ to be \emph{hereditarily preprincipal} if every quasicontraction of $R$ is preprincipal.  In particular, since every component of $R$ is a quasicontraction of $R$,  by our conventions, it follows that  a hereditarily   preprincipal  signed groupoid set is preprincipal.
\begin{corollary} \label{herpreprin}Let $R$ be a signed groupoid set.
Then $R$ corresponds to the tope acycloid of some finite oriented matroid (which is then uniquely determined up to isomorphism) if and only if $R$ is faithful, finite, connected, simply connected, hereditarily preprincipal and antipodal.
\end{corollary}

The following is the main result of this paper.

\begin{theorem}\label{mainthm} (a)  Let  $A$ be the tope  (pre)acycloid in $\mathcal{A}$  of  a simplicial oriented matroid $M$. Then $R:=\SGS(A)$ is a faithful, finite, connected, simply connected,  preprincipal and  complete  signed groupoid set.

(b) Let $R$ be a faithful, finite, connected, simply connected,  preprincipal and complete signed groupoid set. Then $R$ is in $\mathcal{R}$ and   $A:=\PA(R)$ is the tope  (pre)acycloid
of a (uniquely determined)  simplicial oriented matroid $M$.

(c) In either (a) or (b),  the following are equivalent:
\begin{enumerate} \item $R$ is real and principal
\item $R$ is   real and compressed
\item $A$ is a simple acycloid
\item $M$ is a simplicial oriented geometry.
\end{enumerate}

(d) In either (a) or (b), every hypercontraction of $R$ is
 a faithful, finite,  connected, simply connected,  preprincipal and complete signed groupoid set, and so also corresponds
 to a (unique up to isomorphism)  simplicial oriented matroid.
 \end{theorem}
\begin{proof} (a) This follows from Lemma \ref{sgsfacts1}(f).

(b)  By Corollary \ref{bhmain2} and Lemma \ref{JOPlem}, all hypercontractions  of $R$ are  finite, connected, simply connected,  preprincipal,   complete (and therefore  antipodal)  signed groupoid sets. In particular, all quasicontractions of $R$ are in  $\mathcal{R}$ and are preprincipal, so $A$ is the  tope acycloid of an oriented matroid $M$ by  Proposition \ref{handareform}(c). Since $R$ is complete,  Proposition \ref{acyctosgs}(f) implies that $M$ is simplicial.

(c) This follows from Proposition \ref{acyctosgs} and Lemma \ref{matroidisacycloid}.

(d) As observed in the proof of  (b), any hypercontraction of $R$ has all the properties assumed of $R$, so this follows by applying (b) to the hypercontractions.
 \end{proof}
 \begin{example}\label{simphyperconexamp} We give a reformulation of Theorem \ref{mainthm} (d) directly in terms of simplicial oriented matroids, leaving the reader to check details.
 Let $M=(E,*,\cx)$ be a simplicial oriented matroid and $A=(E,*,\mathfrak{T})$ be its tope acycloid (that is, $\mathfrak{T}$ is the set of topes of $M$). Fix a tope $H\in \mathfrak{T}$ and a set $X\subseteq \mathfrak{T}$ of topes. Define $U:=\{H\cap K^{*}\mid  K\in X\}$. Let $\mathfrak{T}':=\{F\in \mathfrak{T}\mid U\subseteq \{F\cap K^{*}\mid  K\in \mathfrak{T}\}\}$ and $L:=\bigcap_{F\in \mathfrak{T}'}F$ (note $L\supseteq  \bigcup_{Y\in U}Y$).
 Finally, let $\mathfrak{S}:=\{F\setminus L\mid F\in \mathfrak{T}'\}$. Then $A\dslash(H,X):=(E,*,\mathfrak{S})$ is the tope acycloid of a uniquely determined simplicial oriented matroid,  which one might denote by $M\dslash(H,X)$.
  \end{example}

The following is a straightforward consequence of Theorems \ref{mainthm} and \ref{closure}.
\begin{corollary}
Let $R=(G,\Phi,\Phi_{+})$ be a finite, faithful, connected, simply connected,  preprincipal and rootoidal signed groupoid set.
Let $M=(\Phi, *,\cx)$ be the finite oriented matroid with tope acyloid
 $A=\PA(R)$.  Then for $X\subseteq \Phi$, one has $$\cx(X)=\tensor*{\Phi}{_{\mathrm{im}}}\cup\Biggl(\, \bigcup_{a\in \Ob  (G)}\,\bigcap_{
 \substack{
                    {b\in \Ob  (G)}\\
            {\tensor*[_{b}]{\Phi}{_{\mathrm{re}}^{+}}
            \supseteq X\cap
            \tensor*[_{a}]{\Phi}{_{\mathrm{re}}^{+}} }     }}
\lrsub{b}{\Phi}{\mathrm{re}}^+\Biggr).$$
\end{corollary}

For the remainder of this subsection,
let $ R= (G,\Phi,\Phi^+)$ be a finite,  connected, simply connected, real,  compressed, hereditarily  preprincipal and  antipodal  signed groupoid set.  By Proposition \ref{handareform},
the corresponding acycloid  $\PA(R)$  is the tope acycloid of an oriented geometry $M=(\Phi,*,\cx)$. We shall give a condition for realizability of $M$ (or more generally, its embeddability in another oriented matroid) involving an analogue of  a standard condition on the relation of simple roots and positive roots of root systems in real vector spaces.  First we check that  two possible  notions of simple roots for $R$ agree, though this is not strictly necessary.

\begin{lemma} Let $S$ be the set of simple morphisms of $G$. Then for any  $a\in \Ob(G)$, one has  $\bigcup_{s\in S\cap \lsub{a}{G}} \Phi_{s}=\ex(\lsub{a}{\Phi}^{+})$.  We denote this set by $\lsub{a}{\Pi}$ and call it the set of simple roots at $a$.\end{lemma}
\begin{proof} We shall use the properties of extreme elements mentioned in \ref{extreme}. Suppose $s\colon b\to a$ is simple.
Write $\Phi_{s}=\{\alpha\}$ where $\alpha\in \Phi$. Then
$\lsub{a}{\Phi}^{+}\cap -\lsub{b}\Phi^{+}=\{\alpha\}$. This implies that $\lsub{a}{\Phi}^{+}\setminus \{\alpha\}=
\lsub{a}{\Phi}^{+}\cap \lsub{b}{\Phi}^{+}$ is $\cx$-closed. So $\alpha\in \ex(\lsub{a}{\Phi}^{+})$, or else we would have
\begin{equation*}
\lsub{a}{\Phi}^{+}= \cx(\ex(\lsub{a}{\Phi}^{+}))\subseteq
\cx(\lsub{a}{\Phi}^{+}\setminus \{\alpha\})=\lsub{a}{\Phi}^{+}\setminus \{\alpha\}
\end{equation*} which is a contradiction.  This proves $\bigcup_{s\in S\cap \lsub{a}{G}} \Phi_{s}\subseteq\ex(\lsub{a}{\Phi}^{+})$.

For the reverse inclusion, let $\alpha\in \ex(\lsub{a}{\Phi}^{+})$.
Then $\cx(\lsub{a}{\Phi}^{+}\setminus\{\alpha\})=\lsub{a}{\Phi}^{+}\setminus\{\alpha\}$ since otherwise, $\cx(\lsub{a}{\Phi}^{+}\setminus\{\alpha\})=\lsub{a}{\Phi}^{+}$, and there would be a minimal subset $\Gamma\subseteq  \lsub{a}{\Phi}^{+}\setminus\{\alpha\}$ with $\cx(\Gamma)=\lsub{a}{\Phi}^{+}$. But then $\Gamma$ would  be a  minimal subset of    $\lsub{a}{\Phi}^{+}$ satisfying
$\cx(\Gamma)=\lsub{a}{\Phi}^{+}$, so $\Gamma=\ex(\lsub{a}{\Phi}^{+})$, contrary to $\alpha\in\ex(\lsub{a}{\Phi}^{+})\setminus \Gamma$. Since $\lsub{a}{\Phi}^{+}\setminus\{\alpha\}$
is a closed  subset of a hemispace, it is an intersection of hemispaces by Theorem \ref{closure}. This implies that $(\lsub{a}{\Phi}^{+}\setminus \{\alpha\})\cup \{\alpha^{*}\}$ is a hemispace,
say equal to $\lsub{b}\Phi$. The unique morphism $s\colon b\to a$
has $\Phi_{s}=\{\alpha\}$. Thus, $\alpha\in \Phi_{s}$ where  $s\in S\cap \lsub{a}{G}$, as required.
\end{proof}

For the proof of the theorem below, we use the following definition.

\begin{definition}We say that a subset $A$ of $\Phi$ is a \emph{half set} if $A\dotcup  A^{*}=\Phi$.  For half sets $A$ and $B$ of $\Phi$,  we define $d(A,B)=\frac{\vert A+B\vert }{2}=\frac{\vert \Phi\vert}{2} -\vert A\cap B\vert $ where $+$ denotes the symmetric difference operation  (that is, $A+B:=(A\cup B)\setminus (A\cap B)$).
\end{definition}

\begin{definition}
Let  $R=(G,\Phi,\Phi^+)$ as above and let  $M'=(E,-,c)$ be any (possibly infinite) oriented matroid.  We define an \emph{embedding} of $R$ in $M'$ to be an injective   map  $f\colon \Phi\rightarrow E$  such that $f(\alpha^{*})=-f(\alpha)$ for any $\alpha\in \Phi$ and one has $c (f(\lsub{a}{\Pi}))\cap f(\Phi)=f(\lsub{a}{\Phi}^+)$ for any $a\in \Ob  (G)$.

If there is real vector space $V$ such  that $M'$ is the standard oriented matroid
$M'=(E, -,c)$ (where $E=V\setminus \{0\}$,  $-x$ is the  additive inverse of $x\in E$ and
$c(X)=\cone(X)\cap E$ for all $X\subseteq E$), we also  call an embedding  $f$ of $R$ in $M'$ a \emph{realization} of $R$ in $V$.\end{definition}

\begin{theorem}  Let $R$ be as  above, $M'=(E,-,c)$ be any (possibly infinite) oriented matroid and $f\colon \Phi\to E$ be an embedding of $R$ in $M'$ in the above sense.  Then $f$ induces an isomorphism of  oriented matroids from $M$ (the oriented matroid associated to $R$) to the restriction of  $M'$ to $f(\Phi)$.
\end{theorem}
\begin{proof}
We use Lemma \ref{firstcondition} for $M$ and $M'$.  We need only verify its hypotheses. The definition of embedding  guarantees that for any hemispace $\lsub{a}{\Phi}^{+}$ of $M$, one has  $c(f(\lsub{a}{\Phi}^+))\cap f(\Phi)=f(\lsub{a}{\Phi}^+)$.  If the hypotheses in  Lemma  \ref{firstcondition} fails, there is  therefore   some half set $A$ of $ \Phi$ such that   $c (f(A))\cap f(\Phi) =c(f(A))$ but $A\neq \lsub{a}{\Phi}^+$ for any $a\in \Ob  (G)$.
Let $\lsub{b}{\Phi}^+$ be such that $d(\lsub{b}{\Phi}^+,A)$ is minimal. We claim that $\lsub{b}{\Pi}\subseteq A$. Suppose that there exists $\beta\in \lsub{b}{\Pi}\setminus A$. Then $(\lsub{b}{\Phi}^+\setminus \{\beta\})\cup\{-\beta\}=\lsub{c}{\Phi}^+$ for some $c\in \Ob  (G)$.
Then $d(\lsub{c}{\Phi}^+, A)=d(\lsub{b}{\Phi}^+, A)-1$. This is a contradiction. So we have established the claim.
 But $f(A)=c (f(A))\cap f(\Phi)\supseteq  c (f(\lsub{b}{\Pi}))\cap f(\Phi)=f(\lsub{b}{\Phi}^+)$. So this forces $f(A)=f(\lsub{b}{\Phi}^+)$ and thus $A=\lsub{b}{\Phi}^+$, which is a contradiction.\end{proof}

 \subsection{} Suppose  $R$ has a realization $f$ in the above sense. Then the associated oriented geometry $M$ is isomorphic to the oriented geometry attached to the set of vectors $f(\Phi)$ in $V$.
 Let  $V_{0}:=\mathrm{span}(f(\Phi))$ in $V$.
 As discussed in Example \ref{coneexamp}, the linear  hyperplanes in $V_{0}$ orthogonal to the elements of $f(\Phi)$ give a real, finite, essential, linear  hyperplane arrangement in $V_{0}$, associated to $M$. In particular,  the chambers of this arrangement correspond bijectively to the topes of $M$ and thus also to the objects of $G$.

 The above all  applies  in particular when $R$ is a finite, connected simply connected, principal and complete signed groupoid set. In that case $M$ is a simplicial  oriented geoemtry and the above hyperplane arrangement is simplicial.

\section{Final comments and open  questions}
\subsection{} In parts of this paper, we  have worked only with connected and simply connected signed groupoid sets.
For many purposes (though of course not, for example,  in parameterizing the components of signed groupoids sets), this does not involve a significant loss of generality:  one can work with the (closely related)  universal covers of the components. In particular, our main  results apply to universal covers of signed groupoid sets attached to finite  Coxeter groups and  Weyl and Coxeter groupoids. It is already well known that  these covers correspond to special
realizable simplicial geometries and to  simplicail
hyperplane arrangements, which have been studied quite deeply in the finite Coxeter group case.

\subsection{} One   consequence of our main theorem
is that it permits a purely algebraic and combinatorial construction   and study of  the underlying oriented matroid of the standard root
 system of a finite Coxeter group, and associated structures, without involving the standard root system in a real vector space, as follows\footnote{A better understanding of such matters is relevant  to  study of root systems of infinite Coxeter groups (for which there is no canonical $W$-stable  oriented matroid structure on the abstract root system, often many non-isomorphic   such structures arising   from  various realized  root systems, and  conjecturally many  more non-realizable such structures, though none are known so far.)}.

 Given a Coxeter system
 $(W,S)$, let $T=\{wsw^{-1}\mid w\in W, s\in S\}$ be the set of abstract reflections.  Regard $\Phi:=T\times\{\pm\}$ as strictly
  involuted set with involution $-(t,\pm)=(t,\mp)$ and $W$-action
  determined by $s(s,\pm)=(s,\mp)$ and   $s(t,\pm)=(sts,\pm)$ if
  $t\neq s$, for $s\in S$ and $t\in T$. We call $\Phi$ the standard abstract root system of $(W,S)$; see    \cite[Ch IV, \S1, no. 4]{Bour} or \cite[1.3]{bjornerbrenti}.
  Define \begin{equation*}
  \text{\rm $\Phi^{+}:=T\times\{+\}$\quad  and \quad
  $\mathfrak{T} :=\{w(\Phi^{+})\mid w\in W\}$.}
  \end{equation*}
   If $W$ is finite (as we assume for now) then
   $A=(\Phi,-,\mathfrak{T})$ is the tope acycloid of a   simplical
   oriented geometry on which $W$ acts as a group of
   automorphisms.

   This is  easy to see using the natural
   identification of $\Phi$ as $W$-set with the standard root
   system of $(W,S)$ in a real vector space, but it is not so
   straightforward to verify otherwise (for example, from Handa's
  (or other)   characterizations of oriented matroids in terms of their topes, or
  in terms of axioms for the oriented  circuits or closure operator as may be defined in terms of  $A$; see \cite{OrMatBook}). Using our theorem, the result follows directly, though we won't give details; the most delicate fact required  is that the weak order of a (finite) Coxeter group is a meet semilattices (which may be proved without recourse to realized  root systems, for instance as in \cite{bjornerbrenti}).

\subsection{} Let $R$ denote the (finite, principal, complete)
signed groupoid set attached above  to the finite Coxeter system
$(W,S)$.  Recall we  defined hypercontractions of $R$ as
componets of the generalized Brink Howlett construction applied
to $R$.  We remark that in general, many  hypercontractions
will be ``trivial'' or ``small'', and many sets of them will all be canonically isomorphic for general reasons related to certain  Galois connections.  Nonetheless,  there is still considerable richness in the class of hypercontractions (for example, they include the components of the original groupoid studied by Brink and Howlett).

Our main theorem implies the hypercontractions all have the same properties as listed for  $R$, so they (more precisely, the real compressions of their universal covers) are isomorphic to signed groupoid sets associated to simplicial oriented geometries. Though we haven't
given details in this paper, there are  more general constructions in categories of  rootoids (functor rootoids
and categorical limits, for instance)   which  may be interpreted as
constructions preserving   the relevant class of  signed groupoid
sets, and so  produce simplicial oriented geometries from families
of simplicial oriented geometries, by the main theorem of this
paper.

  An important  point is that while $R$ is known to be associated to a realizable simplicial oriented geometry, it is not
known whether simplicial oriented geometries associated to its hypercontractions are realizable. Realizability is currently
known to hold only in very special classes of examples, by techniques which do not extend in an obvious way to the general situation. Similar remarks apply
to signed groupoid sets $R$ associated to Coxeter groupoids and indeed to those associated to realizable simplicial oriented geometries in general. Approaches to the study  of Coxeter groupoids and Weyl groupoids in the literature require the existence of a suitable realized root system in order to develop their basic properties, and an abstract construction of the associated signed groupoid set and its corresponding simplicial oriented geometry, as discussed above for finite Coxeter groups, is not currently available.

\subsection{} With these general comments in hand, we list below a few of many questions and problems we leave open in this work. It is quite possible  that simple counterexamples, constructions or arguments could settle  some of them.

\begin{enumerate}
\item Is the simplicial oriented geometry associated to  (the real compression of the) hypercontraction of the signed groupoid set attached to a realizable simplicial oriented geometry itself realizable?  If so, is the analogous statement true for other constructions from \cite{rootoid1}--\cite{rootoid2}.  If not, are there conditions under which realizability is preserved by such constructions, either in general or  for  natural subclasses  such as  signed groupoid sets from Coxeter groups or Weyl or Coxeter groupoids. When realizability is preserved, one has  additional questions of whether  rationality properties are  preserved; if one starts with a crystallographic, in a suitable sense, realized root system for the original signed groupoid set, is there a
natural  crystallographic root system for the sigend groupoid set constructed from it.
\item Develop a general theory of signed groupoids sets enriched by compatible (possibly infinite) oriented matroid structures on their root systems at the various objects.  Use these in particular as a framework for an extended theory of (possibly infinite)   Weyl and Coxeter groupoids.  Study these in particular for (infinite) Coxeter groups with a view to constructing non-realizable examples.

\item Under what conditions are (real compressions of) hypercontractions of  signed groupoid sets from finite Coxeter groups (and from  Weyl or Coxeter groupoids) themselves signed groupoid sets from finite Weyl or Coxeter groupoids (in a suitably generalized sense as in (2)). Similarly for infinite Coxeter groups and groupoids.

\item A faithful, connected, simply connected, antipodal  signed groupoid set can be attached to any (possibly infinite)  oriented matroid, in a very  similar way as for  the case of finite oriented matroids
(see \ref{matroidisacycloid}). Develop a natural  characterization for the class of these signed groupoid sets (the finite ones are characterized by  \ref{herpreprin}). Does the subclass of those, all of whose weak orders are complete lattices, form  a reasonable generalization to infinite ground sets of the class of simplicial oriented  matroids? See \cite{orthogonalgroup} for an example.

\item Are  signed groupoid sets attached to  natural subclasses of finite oriented matroids  stable under (natural subclasses of) hypercontractions?  This paper proves that this holds for the class of  simplicial oriented matroids and all hypercontractions, but it  is open for the class of all finite oriented matroids.  Similarly  for other constructions and for signed groupod sets from infinite oriented matroids.

\item Study hereditarily preprincipal signed groupoids sets as a generalization of the class of signed groupoid sets attached to finite oriented matroids (which forms a subclass of hereditarily preprincipal ones  by \ref{herpreprin}).

\item The combinatorics of squares is enormously rich (see \cite{rootoid2} for some indications of this). However, not much is known about the  detailed combinatorics of squares  in special cases.  It is a natural problem  to classify or describe more concretely the squares (or more generally,  hypercubes: hypercubical diagrams all of the  two dimensional faces of which are squares) in various special settings;  for example, in signed groupoid sets from  symmetric groups, classical Weyl groups,  finite Coxeter groups and Coxeter groupoids, finite simplicial oriented matroids, finite  oriented matroids, general Coxeter groups, general oriented matroids etc.

\item We finish with a more concrete question related to (3) and (7). Suppose that
$R=(G,\Phi,\Phi^{+})$ is the (not  simply connected) signed groupoid set attached to a Coxeter group (or perhaps a Coxeter groupoid) $W$. Consider some hypercontraction $R\dslash(a,X)=(H,\Psi,\Psi^{+})$ and a self-composable simple morphism $s\colon a\to a$ of  $H$. Is  $s$ necessarily an involution (that is,  does it satisfy $s^{2}=1_{a}$)? This would be necessary for  $R\dslash(a,X)$ to come from a Coxeter  groupoid
(with the atomic morphisms of $H$ as its  simple morphisms).
 \end{enumerate}

\section{Acknowledgement}

This paper is based on part of the second author's dissertation under the supervision of the first author.


\begin{thebibliography}{10}

\bibitem{bes}
David Bessis.
\newblock Finite complex reflection arrangements are {$K(\pi,1)$}.
\newblock {\em Ann. of Math. (2)}, 181(3):809--904, 2015.

\bibitem{bjornerbrenti}
Anders Bj\"{o}rner and Francesco Brenti.
\newblock {\em Combinatorics of {C}oxeter groups}, volume 231 of {\em Graduate
  Texts in Mathematics}.
\newblock Springer, New York, 2005.

\bibitem{hyperplane}
Anders Bj\"{o}rner, Paul~H. Edelman, and G\"{u}nter~M. Ziegler.
\newblock Hyperplane arrangements with a lattice of regions.
\newblock {\em Discrete Comput. Geom.}, 5(3):263--288, 1990.

\bibitem{OrMatBook}
Anders Bj\"{o}rner, Michel Las~Vergnas, Bernd Sturmfels, Neil White, and
  G\"{u}nter~M. Ziegler.
\newblock {\em Oriented matroids}, volume~46 of {\em Encyclopedia of
  Mathematics and its Applications}.
\newblock Cambridge University Press, Cambridge, second edition, 1999.

\bibitem{Bour}
Nicolas Bourbaki.
\newblock {\em Lie groups and {L}ie algebras. {C}hapters 4--6}.
\newblock Elements of Mathematics (Berlin). Springer-Verlag, Berlin, 2002.
\newblock Translated from the 1968 French original by Andrew Pressley.

\bibitem{bhaut}
Brigitte Brink and Robert~B. Howlett.
\newblock A finiteness property and an automatic structure for {C}oxeter
  groups.
\newblock {\em Math. Ann.}, 296(1):179--190, 1993.

\bibitem{bh}
Brigitte Brink and Robert~B. Howlett.
\newblock Normalizers of parabolic subgroups in {C}oxeter groups.
\newblock {\em Invent. Math.}, 136(2):323--351, 1999.

\bibitem{largeconvex}
J.~Richard B\"{u}chi and William~E. Fenton.
\newblock Large convex sets in oriented matroids.
\newblock {\em J. Combin. Theory Ser. B}, 45(3):293--304, 1988.

\bibitem{radon}
Raul Cordovil.
\newblock A combinatorial perspective on the non-radon partitions.
\newblock {\em J. Combin. Theory, Ser A}, 38:38--47, 1985.

\bibitem{weylgroupoid2}
M.~Cuntz and I.~Heckenberger.
\newblock Weyl groupoids with at most three objects.
\newblock {\em J. Pure Appl. Algebra}, 213(6):1112--1128, 2009.

\bibitem{weylgroupoid4}
M.~Cuntz and I.~Heckenberger.
\newblock Finite {W}eyl groupoids of rank three.
\newblock {\em Trans. Amer. Math. Soc.}, 364(3):1369--1393, 2012.

\bibitem{weylgroupoid3}
Michael Cuntz and Istv\'{a}n Heckenberger.
\newblock Weyl groupoids of rank two and continued fractions.
\newblock {\em Algebra Number Theory}, 3(3):317--340, 2009.

\bibitem{weylgroupoid6}
Michael Cuntz and Istv\'{a}n Heckenberger.
\newblock Finite {W}eyl groupoids.
\newblock {\em J. Reine Angew. Math.}, 702:77--108, 2015.

\bibitem{del}
Pierre Deligne.
\newblock Les immeubles des groupes de tresses g\'{e}n\'{e}ralis\'{e}s.
\newblock {\em Invent. Math.}, 17:273--302, 1972.

\bibitem{rootoid1}
M.~J. Dyer.
\newblock Groupoids, root systems and weak order {I}.
\newblock {\em \tt arXiv:1110.3217 [math.GR]}, 2011.

\bibitem{rootoid2}
M.~J. Dyer.
\newblock Groupoids, root systems and weak order {iI}.
\newblock {\em \tt arXiv:1110.3657[math.GR]}, 2011.

\bibitem{DyW}
Matthew Dyer.
\newblock On the {W}eak {O}rder of {C}oxeter {G}roups.
\newblock {\em Canad. J. Math.}, 71(2):299--336, 2019.

\bibitem{Ed}
Paul~H. Edelman.
\newblock The lattice of convex sets of an oriented matroid.
\newblock {\em J. Combin. Theory Ser. B}, 33(3):239--244, 1982.

\bibitem{EdJam}
Paul~H. Edelman and Robert~E. Jamison.
\newblock The theory of convex geometries.
\newblock {\em Geom. Dedicata}, 19(3):247--270, 1985.

\bibitem{FL}
Jon Folkman and Jim Lawrence.
\newblock Oriented matroids.
\newblock {\em J. Combin. Theory Ser. B}, 25(2):199--236, 1978.

\bibitem{FH}
Komei Fukuda and Keiichi Handa.
\newblock Antipodal graphs and oriented matroids.
\newblock {\em Discrete Math.}, 111(1-3):245--256, 1993.
\newblock Graph theory and combinatorics (Marseille-Luminy, 1990).

\bibitem{topechar}
Keiichi Handa.
\newblock A characterization of oriented matroids in terms of topes.
\newblock {\em European J. Combin.}, 11(1):41--45, 1990.

\bibitem{weylgroupoid}
I.~Heckenberger.
\newblock The {W}eyl groupoid of a {N}ichols algebra of diagonal type.
\newblock {\em Invent. Math.}, 164(1):175--188, 2006.

\bibitem{weylgroupoid5}
Istv\'{a}n Heckenberger and Volkmar Welker.
\newblock Geometric combinatorics of {W}eyl groupoids.
\newblock {\em J. Algebraic Combin.}, 34(1):115--139, 2011.

\bibitem{weylgroupoid1}
Istv\'{a}n Heckenberger and Hiroyuki Yamane.
\newblock A generalization of {C}oxeter groups, root systems, and {M}atsumoto's
  theorem.
\newblock {\em Math. Z.}, 259(2):255--276, 2008.

\bibitem{Hum}
James~E. Humphreys.
\newblock {\em Reflection groups and {C}oxeter groups}, volume~29 of {\em
  Cambridge Studies in Advanced Mathematics}.
\newblock Cambridge University Press, Cambridge, 1990.

\bibitem{Kac}
Victor~G. Kac.
\newblock {\em Infinite-dimensional {L}ie algebras}.
\newblock Cambridge University Press, Cambridge, 1990.

\bibitem{orthogonalgroup}
Annette Pilkington.
\newblock On the weak order of orthogonal groups.
\newblock {\em Comm. Algebra}, 42(5):1965--1993, 2014.

\bibitem{MatBook}
D.~J.~A. Welsh.
\newblock {\em Matroid theory}.
\newblock Academic Press [Harcourt Brace Jovanovich, Publishers], London-New
  York, 1976.
\newblock L. M. S. Monographs, No. 8.

\end{thebibliography}
\end{document}